\documentclass[11pt,reqno]{amsart}
\usepackage[T1]{fontenc}
\usepackage{lmodern}
\usepackage{amsmath,amssymb,amsthm}
\usepackage[margin=30mm]{geometry}
\usepackage{microtype}
\usepackage[hidelinks]{hyperref}
\hypersetup{pdftitle={Sharp Euler-factor positivity for the BG-rank-refined five-core crank},
  pdfsubject={Cyclotomic positivity and core partitions},
  pdfkeywords={integer partitions, five-core crank, BG-rank, theta series}}
\allowdisplaybreaks[1]
\numberwithin{equation}{section}
\newtheorem{theorem}{Theorem}[section]
\newtheorem{proposition}[theorem]{Proposition}
\newtheorem{lemma}[theorem]{Lemma}
\newtheorem{corollary}[theorem]{Corollary}
\theoremstyle{remark}
\newtheorem{remark}[theorem]{Remark}
\newtheorem{example}[theorem]{Example}
\newcommand{\Z}{\mathbb Z}
\newcommand{\Nzero}{\mathbb Z_{\ge0}}
\newcommand{\BG}{\operatorname{BG}}
\newcommand{\cc}{\operatorname{cc}}
\newcommand{\A}{\mathcal A}
\newcommand{\PP}{\mathcal P}
\newcommand{\RR}{\mathcal R}
\newcommand{\GG}{\mathcal G}

\title[Sharp positivity for the five-core crank]{Sharp Euler-factor positivity for the BG-rank-refined five-core crank}
\author{Huan Xiong}
\address{Institute for Advanced Study in Mathematics, Harbin Institute of Technology,
Heilongjiang 150001, P. R. China}
\email{huan.xiong.math@gmail.com}

\date{}
\subjclass[2020]{Primary 05A17, 11P83; Secondary 05A19}
\keywords{Integer partitions, five-core crank, BG-rank, cyclotomic positivity, theta series, partition inequalities}

\begin{document}
\raggedbottom
\begin{abstract}
In this paper, we complete the resolution of Stanton's problem OPAC--005 by proving its fixed-BG-rank assertion. We further establish a sharp nine-factor positivity theorem for the joint generating function.

Let $K_{n,j}(z)$ enumerate partitions of $5n+4$ with BG-rank $j$ by the integer five-core crank prescribed in OPAC--005, and put $[5]_z=1+z+z^2+z^3+z^4$. We prove that $K_{n,j}(z)/[5]_z$ has nonnegative integer Laurent coefficients for every $n\ge0$ and $j\in\Z$. The refined equidistribution theorem of Berkovich and Garvan gives divisibility; the new assertion is coefficientwise nonnegativity.

Set $\Psi(u,z,q)=\sum_{n\ge0}\sum_{j\in\Z}q^{n+1}u^jK_{n,j}(z)$ and $D(q)=\prod_{d\ge1}(1-q^{2d})$. For every nonnegative integer $M$, we show that $D(q)^M\Psi(u,z,q)/[5]_z$ has nonnegative coefficients if and only if $M\le9$. Thus positivity survives after removing nine of the ten partition factors supplied by the core--quotient decomposition. A coefficient equal to $-4$ at $M=10$ proves sharpness.

The key step is a compensation between lattice terms of different core sizes that preserves both the BG-rank and the full integer crank. Theta identities and three local coefficient inequalities provide the required nonnegative contributions, and an explicit allocation ensures that none is used twice. The resulting factorization also yields a nine-color convolution and coefficientwise finite-difference inequalities through order nine.
\end{abstract}
\maketitle

\section{Introduction}

Stanton's problem OPAC--005 ~\cite{Stanton} asks whether division by the fifth
cyclotomic polynomial gives a nonnegative quotient for a prescribed
integer five-core crank, both without restriction and with the
BG-rank fixed. We prove the fixed-BG-rank assertion and determine
how many Euler factors can be removed while preserving positivity
in both weights. The sharp result has a core--quotient interpretation:
the ten quotient partitions supply ten partition factors, and
retaining one of them suffices for positivity, whereas removing
all ten does not.

Write $\Z$ for the integers and $\Nzero=\{0,1,2,\ldots\}$.
A partition $\lambda=(\lambda_1,\ldots,\lambda_\ell)$ is a finite
weakly decreasing sequence of positive integers. Write
$\lambda\vdash N$ if $|\lambda|=\sum_i\lambda_i=N$, and put
$\lambda_i=0$ for $i>\ell$. The empty partition has size zero.
The Young diagram of $\lambda$ has $\lambda_i$ cells in row $i$;
its conjugate $\lambda'$ records column lengths.
Rows and columns are numbered from one; cell $(i,k)$ has content
$k-i$ and hook length $\lambda_i-k+\lambda'_k-i+1$.
For a positive integer $t$, a partition is a \emph{$t$-core}
if no hook length is divisible by $t$.
For a statement $E$, let $\mathbf1_E$ be its indicator.
The \emph{BG-rank} is
\begin{equation}\label{eq:bg-def}
 \BG(\lambda)=\sum_{(i,k)\in\lambda}(-1)^{i+k}
 =\sum_{i\ge1}(-1)^{i+1}\mathbf1_{\{\lambda_i\text{ odd}\}}.
\end{equation}
Berkovich and Garvan introduced this statistic in their work on the
Andrews--Stanley refinement of Ramanujan's congruence
modulo five~\cite{BG06}; see also~\cite{CJW}.

For $r\in\{0,1,2,3,4\}$, let $R_r(\lambda)$ count the cells whose
contents are congruent to $r$ modulo five. We use the integer statistic
\begin{equation}\label{eq:crank-def}
 \cc(\lambda)=1+2\bigl(R_0(\lambda)-R_4(\lambda)\bigr)
                  +R_1(\lambda)-R_3(\lambda).
\end{equation}
Its residue modulo five is the five-core crank of
Berkovich and Garvan~\cite[(4.1)]{BG08}.
For $|\lambda|\equiv4\pmod5$, it is exactly the integer lift
specified in OPAC--005~\cite{Stanton}, as verified in
Section~\ref{sec:coordinates}; also $\cc(\varnothing)=1$.
The choice of integer lift matters: changing crank values by
multiples of five preserves residue classes but can change the
signs of the quotient coefficients.

Let $p(N)$ count partitions of $N$, and let $p_j(N)$ count those
with BG-rank $j$. Ramanujan's congruence
$p(5n+4)\equiv0\pmod5$ has combinatorial explanations through
Dyson's rank, the Andrews--Garvan crank, and core--quotient
methods~\cite{Dyson,ASD,AG,GKS}.
Berkovich and Garvan~\cite[Theorem~4.1]{BG08} proved that the
five-core crank remains equidistributed modulo five on partitions
of $5n+4$ with any fixed BG-rank $j$.
Thus, for
\begin{equation}\label{eq:K-def}
 K_{n,j}(z)=\sum_{\substack{\lambda\vdash5n+4\\\BG(\lambda)=j}}
                         z^{\cc(\lambda)},
 \qquad [5]_z=1+z+z^2+z^3+z^4,
\end{equation}
with $n\in\Nzero$ and $j\in\Z$, the divisibility
$[5]_z\mid K_{n,j}(z)$ in $\Z[z,z^{-1}]$ is already known.
Indeed, equidistribution gives vanishing at a primitive fifth root
of unity, whose minimal polynomial is $[5]_z$.
It does not imply coefficientwise positivity of the quotient:
\begin{equation}\label{eq:divisibility-example}
 1+z^2+z^3+z^4+z^6=[5]_z(1-z+z^2).
\end{equation}
The OPAC--005 page attributes the unrestricted positive-quotient
assertion to Bringmann, Ono, and Rolen, citing a July 2019
communication from Ono~\cite{Stanton}. We derive the following theorem, which proves the fixed-BG-rank assertion for the specified integer crank, thus completing the whole proof of Stanton's problem
OPAC--005~\cite{Stanton}.

\begin{theorem}[OPAC--005 with fixed BG-rank]\label{thm:main}
For every $n\in\Nzero$ and $j\in\Z$,
\begin{equation}\label{eq:main}
 H_{n,j}(z):=\frac{K_{n,j}(z)}{[5]_z}
        \in\Nzero[z,z^{-1}].
\end{equation}
\end{theorem}

Laurent polynomials allow negative integer powers.
Here $\Nzero[z,z^{-1}]$ denotes those with nonnegative integer
coefficients, including zero when the partition class is empty. For example, $H_{0,0}=1$ and
$H_{1,1}=z^{-2}+2+z^2$; see Example~\ref{ex:cores-only}.

For independent formal variables $u,z,q$, define
\begin{equation}\label{eq:intro-Psi}
 \Psi(u,z,q)=\sum_{n\ge0}q^{n+1}\sum_{j\in\Z}u^jK_{n,j}(z),
 \qquad D(q)=\prod_{d\ge1}(1-q^{2d}).
\end{equation}
Each $q$ coefficient is a finite Laurent polynomial in $u,z$.
The core--quotient theorem gives $\Psi=D^{-10}\Omega$, where
$\Omega$ enumerates ten-core partitions with the same BG-rank
and integer crank weights (see Proposition~\ref{prop:ten-core}).
Each cell of a ten-quotient partition adds two to the rescaled
$q$-degree, explaining the ten copies of $D^{-1}$.
Our next result determines how many copies
can be removed.

\begin{theorem}[Sharp Euler-factor positivity]\label{thm:nine-factors}
For every nonnegative integer $M$,
\begin{equation}\label{eq:sharp-nine-intro}
 D(q)^M\frac{\Psi(u,z,q)}{[5]_z}
 \in\Nzero[u,u^{-1},z,z^{-1}][[q]]
 \quad\Longleftrightarrow\quad M\le9.
\end{equation}
In particular, the series
\begin{equation}\label{eq:G9-def}
 \GG_9(u,z,q)=D(q)^9\frac{\Psi(u,z,q)}{[5]_z}
             =\frac{\Omega(u,z,q)}{[5]_zD(q)},
 \qquad \Omega=D(q)^{10}\Psi,
\end{equation}
has nonnegative integer coefficients.
\end{theorem}

Thus the positive quotient survives with just one of the ten
partition factors. Removing the last factor fails, since (see Proposition~\ref{prop:global-obstruction})
\begin{equation}\label{eq:intro-obstruction}
 [q^7u^0z^0]\,\frac{\Omega(u,z,q)}{[5]_z}=-4.
\end{equation}
This also excludes
every $M\ge10$: positivity for such an $M$ would imply positivity
for $M=10$ after multiplication by $D^{-(M-10)}$.

The first stage of the proof produces a more explicit expansion
when two partition factors are retained.

\begin{theorem}[Eight-factor expansion]\label{thm:strong}
The series
\begin{equation}\label{eq:strong}
 \GG(u,z,q):=D(q)^8\frac{\Psi(u,z,q)}{[5]_z}
       \in\Nzero[u,u^{-1},z,z^{-1}][[q]]
\end{equation}
has the nonnegative lattice expansion in
Theorem~\ref{thm:expansion}.
\end{theorem}

Theorem~\ref{thm:nine-factors} implies Theorem~\ref{thm:main}
and the positivity assertion of Theorem~\ref{thm:strong}.
The additional content of Theorem~\ref{thm:strong} is its explicit
nonnegative lattice expansion, which already proves
Theorem~\ref{thm:main} after multiplication by $D^{-8}$.

\subsection{Proof strategy}

The proof first expresses $\Omega$ as a weighted lattice sum
using the classical core--quotient formulas. The coordinates keep
the exact integer crank exponent and the signs with which the
five quotient partitions contribute to BG-rank.
A rotation identity establishes division by $[5]_z$ within the
Laurent-polynomial coefficient ring. An involution then pairs
lattice terms of different core sizes
(see Section~\ref{sec:pairing}). Two theta telescoping identities make
each paired term nonnegative after division by $D^2$, giving the
eight-factor expansion.

With only one copy of $D^{-1}$, individual paired terms need not
be nonnegative. Section~\ref{sec:one-factor} classifies the
remaining exceptions: they have last lattice coordinate one and
one of two parity patterns (see Corollary~\ref{cor:exceptional-pairs}).
The ninth-factor proof compensates these terms using contributions
from other lattice vectors. Three local inequalities in
\eqref{eq:three-local-compensations} supply the required
coefficientwise comparisons in the BG variable.
The translations in Section~\ref{sec:nine-factors} preserve the
full crank exponent, including the chosen monomial of the finite
geometric factor in the paired expansion. Their six target classes
are disjoint, so no positive contribution is used twice.
This disjointness, together with the quadratic coordinate bound,
also makes the regrouping locally finite.

The local inequalities require simultaneous control of all negative
contributions at a given weight. Quartic theta identities locate
them; coefficient bounds and explicit injections provide the total
capacity needed, including weights at which several negative
sources share a target. The finite rational identities in
Appendix~\ref{app:l1-certificates} support these infinite
coefficient arguments. The resulting coefficient interpretation
does not provide a natural five-cycle on Young diagrams.

\subsection{Consequences and related work}\label{subsec:attribution}

Bringmann, Gomez, Rolen, and Tripp~\cite{BGRT} proved
Stanton-type positivity for ordinary and modified Andrews--Garvan
crank polynomials using symmetry, unimodality and related coefficient
inequalities, together with cyclotomic divisibility. Here unimodality
means weak increase followed by
weak decrease in the coefficient sequence. They explicitly note
in~\cite[Section~1.3]{BGRT} that their methods do not establish
the five-core crank conjecture discussed there because that crank
appears not to be unimodal. Our argument uses compensation across
lattice terms. Moreover, the centered numerator of a specific slice
in the ninth-factor theorem is not unimodal
(see Remark~\ref{rem:nine-unimodality}), so their
Lemma~3.1 does not apply directly to that slice.
Bringmann, Man, and Rolen~\cite{BMR} proved the rank unimodality
statement used for modified-rank positivity; their statistic is
Dyson's rank. Folsom, Males, and Rolen~\cite[Theorem~3.9]{FMR}
derive cyclotomic divisibility for crank polynomials of core
partitions from known equidistribution. The assertion proved here
is positivity for the specified integer five-core crank on all
partitions of $5n+4$, with BG-rank fixed.

The core--quotient and two-core generating functions used here
are classical~\cite{GKS,BG06,BG08,CJW}. The symmetries in
Section~\ref{subsec:affine-symmetries} specialize Fayers'
action~\cite{Fayers14,Fayers23} and reduce the distributions to
BG-ranks $-1,0,1$. Generating functions for cores with prescribed
BG-rank and their weighted or bounded variants appear
in~\cite{BG08,BY,HJ,BD25,Cho25}; related counting formulas occur
in~\cite{CLSXY24,LSSY26}. For strict partitions, whose parts are
distinct, Remark~\ref{rem:strict-partition-kernel} specializes
\cite[Theorem~3.1]{BU16}. In the present problem the quotient
partitions contribute to BG-rank: five-cores of size nine alone
have a quotient with a negative coefficient
(see Example~\ref{ex:cores-only}).

The local proofs use the classical quartic theta identity
of~\cite[Lemma~2.2, (2.33)]{BGY19}; the product identity
\eqref{eq:l1-pair-absorption} is the $L=1$ case of
\cite[Proposition~4.3]{BG08eta}. The positive factorization yields
cumulative crank inequalities, a nine-color convolution, and
mixed finite-difference inequalities in
Section~\ref{sec:consequences}. These are consequences of the
same positivity theorem.

\subsection{Organization and notation}

Section~\ref{sec:common-theta} proves the common theta identities,
and the next three sections establish the local inequalities used
in Section~\ref{sec:nine-factors}.
Section~\ref{sec:consequences} collects the resulting inequalities
and convolution formulas. Appendices~\ref{app:l1-certificates}
and~\ref{app:frobenius} give the finite rational certificates and
a Frobenius-coordinate expression for the joint series, respectively.

All variables are formal and all series inequalities are
coefficientwise. For a series $F$, $[q^nu^jz^k]F$ denotes the
coefficient of $q^nu^jz^k$; variables not displayed are left
unextracted. We write $F=G+O(q^d)$ when $F-G$ has no term of
$q$-degree less than $d$. For a one-variable series $F(x)$, put
$F_n=[x^n]F(x)$, with $F_n=0$ for $n<0$.
We work in $\Z[u,u^{-1},z,z^{-1}][[q]]$ and its analogues with
fewer Laurent variables, so each $q$ coefficient has finite Laurent
support. A family of summands is \emph{locally finite} if only
finitely many summands can contribute through any prescribed
$q$ degree. We verify this before regrouping lattice and theta sums.

We use the $q$-Pochhammer notation
\[
 (a;x)_m=\prod_{r=0}^{m-1}(1-ax^r),\qquad
 (a;x)_\infty=\prod_{r\ge0}(1-ax^r),
 \qquad m\in\Nzero,
\]
with empty product $1$, and
$(a_1,\ldots,a_s;x)_\infty=\prod_{i=1}^s(a_i;x)_\infty$.
In particular, $D(q)=(q^2;q^2)_\infty$.
Euler products with constant coefficient one are units in the
formal-power-series ring. The polynomial $[5]_z$ is not a unit
in $\Z[z,z^{-1}]$; division by it is justified by the explicit
factorization in Section~\ref{sec:pairing}.

\section{Core coordinates and the BG-weighted series}\label{sec:coordinates}

We use the core--quotient conventions of~\cite{GKS,BG08,Fayers14}
and the two-core formula of~\cite{CJW}, fixing the integer crank
and BG signs below.

\subsection{Charges and quotient signs}

For a partition $\lambda$, extended by zero parts, its \emph{beta-set} is
\[
 \mathcal B_\lambda=\{\lambda_i-i:i\ge1\}.
\]
For any set $B$ differing finitely from $\Z_{<0}$, define
\[
 \operatorname{ch}(B)=\#(B\cap\Nzero)-\#(\Z_{<0}\setminus B).
\]
If $B=\{b_1>b_2>\cdots\}$, then $b_i+i=\operatorname{ch}(B)$
eventually; hence charge-zero sets correspond to partitions by
$\lambda_i=b_i+i$.

Fix an integer $h\ge1$. Runner $r$ consists of positions $r+hk$
($0\le r<h$), with level $k$. Set
\[
 B_r=\{k:r+hk\in\mathcal B_\lambda\},\qquad
 c_r=\operatorname{ch}(B_r).
\]
Then $\sum_r c_r=0$, and the occupied levels have the unique form
\begin{equation}\label{eq:levels}
 k_{r,i}=c_r+\mu_{r,i}-i\qquad(i\ge1),
\end{equation}
where $\mu_r$ is a partition. Conversely, any such charges and
partitions determine $\lambda$. The tuple $(\mu_0,\ldots,\mu_{h-1})$
is its \emph{$h$-quotient}; setting every $\mu_r=\varnothing$
gives its \emph{$h$-core} $\kappa(c)$, with occupied levels $k<c_r$.

Indeed, the hook at $(i,k)$ has length
$\lambda_i-k+\lambda'_k-i+1$, where
$\lambda'_k=\#\{i:\lambda_i\ge k\}$: the difference between the
occupied position $\lambda_i-i$ and the unoccupied position
$k-1-\lambda'_k$. Thus a hook length divisible by $h$ corresponds to an occupied
position and a smaller gap on one runner. The cores just defined
are exactly the partitions with no such hook.

The size formulas are
\begin{equation}\label{eq:core-size}
 |\kappa(c)|=\frac h2\sum_{r=0}^{h-1}c_r^2+
                         \sum_{r=0}^{h-1}r c_r,
 \qquad
 |\lambda|=|\kappa(c)|+h\sum_{r=0}^{h-1}|\mu_r|.
\end{equation}
Relative to the empty partition, runner $r$ contributes
$r c_r+h c_r(c_r-1)/2$ to the core size. Summing uses
$\sum_r c_r=0$, and \eqref{eq:levels} adds $h|\mu_r|$.
For $c_r<0$, the same calculation uses the directed-sum convention
$\sum_{k=0}^{c_r-1}=-\sum_{k=c_r}^{-1}$.

Summing checkerboard signs row by row gives
\begin{equation}\label{eq:beta-bg}
 \BG(\lambda)=\frac12\sum_{i\ge1}
       \left((-1)^{\lambda_i-i}-(-1)^{-i}\right).
\end{equation}
For $h=5$, define
\begin{equation}\label{eq:sigma-beta}
 \sigma_r(c)=(-1)^{r+c_r},\qquad
 \beta(c)=\frac{1-\sum_{r=0}^{4}\sigma_r(c)}4.
\end{equation}
The five signs have product $1$, so there are $0$, $2$, or $4$
negative signs and $\beta(c)\in\{-1,0,1\}$. The BG decomposition is
\begin{equation}\label{eq:bg-decomposition}
 \BG(\lambda)=\beta(c)+\sum_{r=0}^{4}\sigma_r(c)\BG(\mu_r).
\end{equation}
To verify the signs, \eqref{eq:beta-bg} gives the core contribution
$\frac12\sum_r(-1)^r\sum_{k=0}^{c_r-1}(-1)^k=\beta(c)$,
and the displacement on runner $r$ contributes
\[
 \frac12\sum_{i\ge1}
 \left((-1)^{r+5(c_r+\mu_{r,i}-i)}-(-1)^{r+5(c_r-i)}\right)
 =\sigma_r(c)\BG(\mu_r).
\]
This is \cite[(3.2)--(3.4)]{BG08} in our convention.

The ordinary BG-rank series \cite[(7.4)]{BG06}, equivalently
the two-core formula \cite[Theorem 1]{CJW}, is
\begin{equation}\label{eq:P-def}
 \PP_+(u,q)=\sum_\nu u^{\BG(\nu)}q^{|\nu|}
       =\frac{S(u,q)}{D(q)^2},\qquad
 \PP_-(u,q)=\PP_+(u^{-1},q),
\end{equation}
including the empty partition, where
\begin{equation}\label{eq:S-D-def}
 S(u,q)=\sum_{j\in\Z}u^j q^{j(2j-1)},\qquad
 D(q)=(q^2;q^2)_\infty.
\end{equation}
Indeed, a two-core has charges $(j,-j)$, size $j(2j-1)$,
and BG-rank $j$ by \eqref{eq:core-size} and \eqref{eq:beta-bg}.
Its two quotient partitions make even bead displacements, preserving
BG-rank and contributing $D(q)^{-2}$. If $p^{(2)}(m)$ counts
ordered pairs of partitions of total size $m$, with $p^{(2)}(m)=0$
for $m\notin\Nzero$, this also gives
\begin{equation}\label{eq:bg-pairs}
 p_j(N)=p^{(2)}\!\left(\frac{N-j(2j-1)}2\right).
\end{equation}

\subsection{Affine coordinates and the integer crank}

Let
\begin{equation}\label{eq:A-QL}
 \A=\left\{a=(a_0,\ldots,a_4)\in\Z^5:\sum_{i=0}^4a_i=1\right\},
 \quad Q(a)=\sum_{i=0}^4(a_i^2-a_i a_{i+1}),
 \quad L(a)=\sum_{i=0}^4 i a_i,
\end{equation}
with subscripts modulo $5$. The map of~\cite[(4.10)--(4.14)]{BG08}
\begin{equation}\label{eq:charge-map}
 \begin{aligned}
 c_0&=a_0+a_4,& c_1&=-a_0+a_1+a_4,& c_2&=-a_1+a_2,\\
 c_3&=-a_2+a_3-a_4,& c_4&=-a_3-a_4
 \end{aligned}
\end{equation}
is a bijection from $\A$ to the charges of five-cores of size
$4$ modulo $5$, with inverse
\begin{equation}\label{eq:inverse}
 \begin{aligned}
 a_4&=(4c_0+3c_1+2c_2+c_3-1)/5,& a_0&=c_0-a_4,\\
 a_1&=c_0+c_1-2a_4,& a_2&=a_1+c_2,& a_3&=-c_4-a_4.
 \end{aligned}
\end{equation}
Indeed, $\sum_r c_r=0$ implies that $\sum_r c_r^2$ is even and
$4c_0+3c_1+2c_2+c_3=-\sum_r r c_r$. Thus \eqref{eq:core-size}
makes \eqref{eq:inverse} integral exactly in the stated size class;
its entries sum to $1$ and recover $c$ under \eqref{eq:charge-map}.
Direct expansion gives
\[
 \frac52\sum_{r=0}^4c_r(a)^2+\sum_{r=0}^4rc_r(a)-(5Q(a)-1)
       =(5a_4-1)\left(\sum_{i=0}^4a_i-1\right),
\]
and hence
\begin{equation}\label{eq:size-Q}
 |\kappa(c(a))|=5Q(a)-1.
\end{equation}

For the cell counts in \eqref{eq:crank-def}, the runner charges satisfy
\begin{equation}\label{eq:residue-charges}
 c_r=R_r(\lambda)-R_{r+1}(\lambda)\qquad(r\bmod5).
\end{equation}
Both sides vanish on $\varnothing$. Adding a cell of content $m$
moves a bead from $m-1$ to $m$, increasing charge on runner
$m\bmod5$ and decreasing it on runner $m-1\bmod5$, even across
level zero. The cell-count differences change identically,
proving \eqref{eq:residue-charges} by induction.
Thus $\cc(\lambda)=1+c_1+c_2-2c_4$ for every partition.
For $|\lambda|\equiv4\pmod5$, substitution of its core coordinate gives
the exact integer identity
\begin{equation}\label{eq:integer-crank}
 \cc(\lambda)=1+c_1+c_2-2c_4
 =1-a_0+a_2+2a_3+3a_4=L(a).
\end{equation}

Abbreviate $\sigma_r(a)=\sigma_r(c(a))$ and
$\beta(a)=\beta(c(a))$, set $\PP_{+1}=\PP_+$, $\PP_{-1}=\PP_-$,
and define
\begin{equation}\label{eq:Phi-Psi}
 \Phi_a(u,q)=u^{\beta(a)}\prod_{r=0}^4\PP_{\sigma_r(a)}(u,q),
 \qquad
 \widetilde\Psi(u,z,q)=\sum_{a\in\A}q^{Q(a)}z^{L(a)}\Phi_a(u,q).
\end{equation}
Equations~\eqref{eq:core-size}, \eqref{eq:bg-decomposition},
\eqref{eq:size-Q}, and \eqref{eq:integer-crank} give
\begin{equation}\label{eq:Psi-enumerates}
 \widetilde\Psi(u,z,q)=\sum_{n\ge0}q^{n+1}\sum_{j\in\Z}u^j K_{n,j}(z)
                       =\Psi(u,z,q).
\end{equation}
We henceforth write $\Psi$; its $q$-degree is $Q(a)$ plus the
total five-quotient size, namely $(|\lambda|+1)/5$.

For local finiteness, observe that
\begin{equation}\label{eq:Q-finite}
 Q(a)=\frac12\sum_{i=0}^4(a_i-a_{i+1})^2\ge1:
\end{equation}
zero would force equal integer coordinates summing to $1$.
Put $d_i=a_i-a_{i+1}$. Cyclic summation gives
\[
 5a_i-1=2d_i+d_{i+1}-d_{i+3}-2d_{i+4},
\]
so Cauchy--Schwarz and $\sum_i d_i^2=2Q(a)$ imply
\begin{equation}\label{eq:coordinate-bound}
 |a_i-1/5|\le 2\sqrt{Q(a)/5}.
\end{equation}
Bounded $Q$ therefore allows finitely many coordinates; at each such
coordinate, fixed degree allows finitely many quotient partitions.
Every $q$-coefficient of $\Psi$ consequently has finite Laurent support.

\subsection{Rotation invariance of the BG weight}

The rotation used in~\cite[(4.17)--(4.18)]{BG08} is
\begin{equation}\label{eq:rotation}
 \rho(a_0,a_1,a_2,a_3,a_4)=(a_4,a_0,a_1,a_2,a_3).
\end{equation}
The definitions of $Q,L$ and $\sum_i a_i=1$ give
\begin{equation}\label{eq:rotation-QL}
 Q(\rho a)=Q(a),\qquad L(\rho a)=L(a)+1-5a_4.
\end{equation}
Put $e_i(a)=(-1)^{a_i+a_{i+1}}$. Reducing
\eqref{eq:charge-map} modulo $2$ yields
\begin{equation}\label{eq:edge-signs}
 (\sigma_0(a),\ldots,\sigma_4(a))=(e_4(a),e_2(a),e_1(a),e_0(a),e_3(a)).
\end{equation}
Rotation permutes these edge signs, preserving $\beta$ and the
product in \eqref{eq:Phi-Psi}; hence
\begin{equation}\label{eq:rotation-Phi}
 \Phi_{\rho a}(u,q)=\Phi_a(u,q).
\end{equation}
If $k(a)$ is the number of negative edge signs, then
\begin{equation}\label{eq:three-weights}
 \Phi_a=u^{k(a)/2-1}\PP_+^{5-k(a)}\PP_-^{k(a)},
 \qquad k(a)\in\{0,2,4\},
\end{equation}
and
\begin{equation}\label{eq:Phi-numerator}
 \Phi_a(u,q)=
 \frac{u^{k(a)/2-1}S(u,q)^{5-k(a)}S(u^{-1},q)^{k(a)}}{D(q)^{10}}.
\end{equation}
Define
\begin{equation}\label{eq:Theta-Omega}
 \Theta_a(u,q)=D(q)^{10}\Phi_a(u,q),\qquad
 \Omega(u,z,q)=\sum_{a\in\A}q^{Q(a)}z^{L(a)}\Theta_a(u,q).
\end{equation}
Thus $\Omega=D^{10}\Psi$, $\Theta_{\rho a}=\Theta_a$, and
both $\Theta_a$ and $\Omega$ have nonnegative coefficients.
The series in Theorem~\ref{thm:strong} is
\begin{equation}\label{eq:G-Omega}
 \GG=\frac{\Omega}{[5]_zD(q)^2}.
\end{equation}
Divisibility by $[5]_z$ is proved in Section~\ref{sec:pairing}.

\subsection{The ten-core numerator}

Let $\mathcal C_{10}$ denote the ten-core partitions.
The following case $(r,s)=(2,5)$ of~\cite[Lemma~2.5(1)]{Fayers14}
includes the BG and integer crank weights.

\begin{proposition}[A specialization of the core--quotient theorem]\label{prop:ten-core}
A partition with five-quotient $(\mu_0,\ldots,\mu_4)$ is a ten-core
if and only if every $\mu_r$ is a two-core. Consequently,
\begin{equation}\label{eq:ten-core-numerator}
 \Omega(u,z,q)=
 \sum_{\substack{\kappa\in\mathcal C_{10}\\|\kappa|\equiv4\pmod5}}
 q^{(|\kappa|+1)/5}u^{\BG(\kappa)}z^{\cc(\kappa)}.
\end{equation}
Passage to the ten-core preserves both BG-rank and the integer
five-core crank.
\end{proposition}

\begin{proof}
A bead--gap pair at distance divisible by ten lies on one five-runner
at even level distance. Translating by its charge gives the
corresponding two-core criterion for the quotient.
Restricting each quotient to a two-core therefore replaces
$\PP_\varepsilon$ by $S(u^\varepsilon,q)$ in \eqref{eq:Phi-Psi},
proving \eqref{eq:ten-core-numerator}.

If $d_0,\ldots,d_9$ are the ten-runner charges, the five-runner charges are
\begin{equation}\label{eq:ten-five-charges}
 c_r=d_r+d_{r+5}\qquad(0\le r<5).
\end{equation}
These are unchanged on passage to the ten-core, as is the integer
crank $1+c_1+c_2-2c_4$. Parity is constant on a ten-runner, so
\eqref{eq:beta-bg} and $\sum_r d_r=0$ give
\begin{equation}\label{eq:ten-bg}
 \BG(\lambda)=\frac12\sum_{r=0}^{9}(-1)^r d_r
             =\sum_{\substack{0\le r<10\\r\text{ even}}}d_r.
\end{equation}
Thus BG-rank is also preserved. Each ten-quotient cell adds ten
cells, or two to the rescaled $q$-degree, giving the ten-quotient
factor $D(q)^{-10}$ and $\Psi=D^{-10}\Omega$.
\end{proof}

\section{A positive theta kernel}\label{sec:kernel}

The sign changes in the lattice pairing lead to a kernel built
from two copies of $S$. For $\ell,t\in\Nzero$, define
\begin{equation}\label{eq:general-kernel}
 \mathcal F_{\ell,t}(u,q)
   =\frac{uS(u^{-1},q)^2-q^{2t+1}S(u,q)^2}{D(q)^\ell}.
\end{equation}
The ordinary two-partition difference is
\begin{equation}\label{eq:F-def}
 F_t(u,q)=u\PP_-(u,q)^2-q^{2t+1}\PP_+(u,q)^2
         =\mathcal F_{4,t}(u,q).
\end{equation}

For a formal variable $x$, put
\begin{equation}\label{eq:theta-def}
 \vartheta_0(x)=\sum_{r\in\Z}x^{2r^2},\qquad
 \vartheta_1(x)=\sum_{r\in\Z}x^{2r(r+1)}.
\end{equation}
If $\ell\ge2$, $d\ge1$, and $r\in\{1,2\}$, define
\begin{equation}\label{eq:R-def}
 \RR^{(\ell)}_{d,r}(x)
  =\frac{(1-x^d)^r}{(x;x)_\infty^\ell}
  =\frac{1}{(1-x^d)^{\ell-r}}
       \prod_{\substack{k\ge1\\k\ne d}}\frac1{(1-x^k)^\ell},
 \qquad
 U_t^{(\ell)}(x)=
 \begin{cases}0,&t=0,\\\RR^{(\ell)}_{t,1}(x),&t\ge1.\end{cases}
\end{equation}
These series have nonnegative coefficients. In a partition in
$\ell$ colors, each part has one of $\ell$ color labels; parts
of the same size and color are indistinguishable. The product
$\RR^{(\ell)}_{d,r}$ counts such partitions, with a
prescribed set of $r$ colors forbidden for parts of size $d$.

\begin{proposition}\label{prop:kernel}
For every $\ell\ge2$ and $t\in\Nzero$,
$\mathcal F_{\ell,t}(u,q)\in\Nzero[u,u^{-1}][[q]]$.
More precisely, for $J\in\Z$ and $x=q^2$,
\begin{equation}\label{eq:kernel-positive}
 [u^J]\mathcal F_{\ell,t}(u,q)=
 \begin{cases}
 q^{J^2-J+1}\displaystyle\left\{
    \sum_{m\ge1}x^{2m(m-1)}\RR^{(\ell)}_{2m,2}(x)
           +\vartheta_0(x)U_t^{(\ell)}(x)\right\},&J\text{ even},\\[6pt]
 q^{J^2-J}\displaystyle\left\{
    \sum_{m\ge0}x^{2m^2}\RR^{(\ell)}_{2m+1,2}(x)
           +x\vartheta_1(x)U_t^{(\ell)}(x)\right\},&J\text{ odd}.
 \end{cases}
\end{equation}
The threshold $\ell=2$ is sharp for positivity of this family for
all $t$: for $\ell\in\Nzero$,
\begin{equation}\label{eq:kernel-sharp}
 [u^1q^2]\mathcal F_{\ell,0}(u,q)=\ell-2.
\end{equation}
\end{proposition}

\begin{proof}
In $[u^J]S(u,q)^2$, write the two summation indices as $h$ and $J-h$.
Completing the square gives
\[
 2h^2-h+2(J-h)^2-(J-h)=4(h-J/2)^2+J^2-J.
\]
Consequently,
\begin{equation}\label{eq:S-coefficient}
 [u^J]S(u,q)^2=q^{J^2-J}
 \begin{cases}
 \vartheta_0(x),&J\text{ even},\\
 q\vartheta_1(x),&J\text{ odd}.
 \end{cases}
\end{equation}
For $[u^J]uS(u^{-1},q)^2$, the two indices have sum $1-J$.
Replacing $J$ by $1-J$ leaves the quadratic prefactor $q^{J^2-J}$
unchanged and reverses the parity. After removing this prefactor, the numerator
in~\eqref{eq:general-kernel} therefore has coefficient
\begin{equation}\label{eq:kernel-numerator}
 \begin{cases}
 q\{\vartheta_1-\vartheta_0+(1-x^t)\vartheta_0\},&J\text{ even},\\
 \vartheta_0-x\vartheta_1+x(1-x^t)\vartheta_1,&J\text{ odd},
 \end{cases}
\end{equation}
where both theta series are evaluated at $x$.

The two differences telescope:
\begin{align}
 \vartheta_1(x)-\vartheta_0(x)
   &=\sum_{m\ge1}x^{2m(m-1)}(1-x^{2m})^2,
      \label{eq:theta-difference-1}\\
 \vartheta_0(x)-x\vartheta_1(x)
   &=\sum_{m\ge0}x^{2m^2}(1-x^{2m+1})^2.
      \label{eq:theta-difference-2}
\end{align}
For the finite partial sums, write
\[
 E_M=1+2\sum_{m=1}^M x^{2m^2},\qquad
 A_M=2\sum_{m=0}^M x^{2m(m+1)}.
\]
For $M\ge1$, the first right-hand sum through $m=M$ is
$A_{M-1}-E_M+x^{2M(M+1)}$. For $M\ge0$, the second right-hand
sum through $m=M$ is $E_M-xA_M+x^{2(M+1)^2}$.
The orders of the boundary terms tend to infinity, proving the
identities in the formal power-series ring.

Divide~\eqref{eq:kernel-numerator} by
$D(q)^\ell=(x;x)_\infty^\ell$. Each square in
\eqref{eq:theta-difference-1}--\eqref{eq:theta-difference-2} gives
$\RR^{(\ell)}_{d,2}$, while
$(1-x^t)/(x;x)_\infty^\ell=U_t^{(\ell)}$ for $t\ge1$.
For $t=0$, this numerator is zero, as required by the definition of
$U_0^{(\ell)}$. This gives~\eqref{eq:kernel-positive}.
All the displayed factors are nonnegative when $\ell\ge2$.
Also $J^2-J\ge0$ for every integer $J$, and the quadratic starting
degrees in $J$ and $m$ ensure local finiteness.

Finally, at $J=1$ and $t=0$, the numerator coefficient is
$\vartheta_0(x)-x\vartheta_1(x)=1-2x+O(x^2)$, whereas
$(x;x)_\infty^{-\ell}=1+\ell x+O(x^2)$.
Their product has coefficient $\ell-2$ at $x=q^2$.
This proves~\eqref{eq:kernel-sharp} and the sharpness assertion.
\end{proof}

\subsection{Support of the two-partition difference}

Let $B(N,J)$ count ordered pairs $(\mu,\nu)$ of partitions with
$|\mu|+|\nu|=N$ and $\BG(\mu)+\BG(\nu)=J$; set $B(N,J)=0$
for $N<0$. The case $\ell=4$ has an exact support description.

\begin{corollary}\label{cor:kernel-support}
For $N,t\in\Nzero$ and $J\in\Z$,
\begin{equation}\label{eq:bipartition-inequality}
 B(N,1-J)-B(N-2t-1,J)\ge0.
\end{equation}
Put
\[
 \eta(J)=J^2-J+\mathbf1_{\{J\text{ even}\}}.
\]
The inequality is strict exactly when
$N\ge\eta(J)$ and $N\equiv\eta(J)\pmod2$.
Equivalently, these are the exact $(q,u)$-support conditions for
$F_t=\mathcal F_{4,t}$, independently of $t$.
\end{corollary}

\begin{proof}
The left side of~\eqref{eq:bipartition-inequality} is
$[q^Nu^J]F_t$. Formula~\eqref{eq:kernel-positive} proves its
nonnegativity and the necessary size and parity conditions.
The first summand is $q^{\eta(J)}\RR^{(4)}_{2,2}(q^2)$ for even
$J$ and $q^{\eta(J)}\RR^{(4)}_{1,2}(q^2)$ for odd $J$.
Both products still allow parts of size one, so they have a
positive coefficient at every nonnegative power of $q^2$.
This proves strict positivity at every permitted degree.
\end{proof}

\begin{remark}[Distinct-part interpretation]\label{rem:strict-partition-kernel}
Let $B_{\mathrm{str}}(N,J)$ count ordered pairs of strict partitions,
that is, partitions with distinct parts, with total size $N$ and total
BG-rank $J$. Set it to zero for $N<0$. The strict-partition BG series is
\begin{equation}\label{eq:strict-bg-known}
 \sum_{\lambda\text{ strict}}u^{\BG(\lambda)}q^{|\lambda|}
 =\frac{S(u,q)}{D(q)}.
\end{equation}
This is the coefficientwise limit of the bounded formula in
\cite[Theorem~3.1]{BU16}; see also~\cite[Theorem~5.3]{LSSY26}.
At each fixed degree only finitely many $J$ occur, so the limit
commutes with summation over $J$. Proposition~\ref{prop:kernel} with $\ell=2$
therefore gives
\[
 B_{\mathrm{str}}(N,1-J)-B_{\mathrm{str}}(N-2t-1,J)\ge0
 \qquad(N,t\in\Nzero,\ J\in\Z).
\]
\end{remark}

\section{A positive lattice factorization}\label{sec:pairing}

Define the involution
\begin{equation}\label{eq:T}
 T(a_0,a_1,a_2,a_3,a_4)
   =(a_0,a_1+a_4,a_2+a_4,a_3,-a_4)
\end{equation}
on $\A$. It exchanges positive and negative last coordinates and
fixes those with last coordinate zero. For $s\ge1$, write
$[s]_{z^5}=\sum_{v=0}^{s-1}z^{5v}$.

\begin{theorem}\label{thm:expansion}
For $b\in\A$ with $s=b_4>0$, put
\begin{equation}\label{eq:paired-kernel}
 \mathcal B_b(u,q)
   =\frac{\Theta_b(u,q)-q^s\Theta_{Tb}(u,q)}{D(q)^2}.
\end{equation}
Then $\mathcal B_b\in\Nzero[u,u^{-1}][[q]]$, and
\begin{equation}\label{eq:positive-lattice}
 \GG(u,z,q)=\frac{\Omega(u,z,q)}{[5]_zD(q)^2}
 =\sum_{\substack{b\in\A\\b_4>0}}
     q^{Q(b)}z^{L(\rho b)}[b_4]_{z^5}\mathcal B_b(u,q).
\end{equation}
The sum is locally finite. In particular, it is a nonnegative
Laurent-series expansion of the series in
Theorem~\ref{thm:strong}.
\end{theorem}

\begin{proof}
We first prove the divisibility identity before dividing by the
Euler product. For an integer $v$, let
\[
 g_v(z)=
 \begin{cases}
 \sum_{i=0}^{v-1}z^{5i},&v>0,\\
 0,&v=0,\\
 -\sum_{i=v}^{-1}z^{5i},&v<0.
 \end{cases}
\]
This is a finite Laurent polynomial, with
$(1-z^5)g_v(z)=1-z^{5v}$. Form the locally finite series
\begin{equation}\label{eq:Xi}
 \Xi(u,z,q)=\sum_{a\in\A}q^{Q(a)}\Theta_a(u,q)
                      z^{L(\rho a)}g_{a_4}(z).
\end{equation}
After multiplication by $1-z^5$, the first part of the sum reindexes
to $\Omega$, by rotation invariance of $Q$ and $\Theta_a$.
The second part is $z\Omega$, because
$L(\rho a)+5a_4=L(a)+1$. Hence
\[
 (1-z^5)\Xi=(1-z)\Omega.
\]
The coefficient ring is an integral domain, so cancellation of
$1-z$ gives
\begin{equation}\label{eq:signed-divisibility}
 [5]_z\Xi=\Omega.
\end{equation}
Thus $\Omega/[5]_z=\Xi$ belongs to the formal-series ring.

The map $T$ preserves the coordinate sum and satisfies $T^2=1$.
Expanding the quadratic and linear forms for an arbitrary vector
gives
\[
 Q(Ta)-Q(a)=a_4\sum_{i=0}^4a_i,\qquad
 L(\rho Ta)-L(\rho a)=5a_4.
\]
Consequently, on $\A$,
\begin{equation}\label{eq:pair-degrees}
 Q(Ta)=Q(a)+a_4,\qquad L(\rho Ta)=L(\rho a)+5a_4.
\end{equation}
The terms with $a_4=0$ vanish. If $b_4=s>0$ and $a=Tb$, then
$a_4=-s$, $Q(a)=Q(b)+s$, and
$L(\rho a)=L(\rho b)+5s$. Since
$g_{-s}=-z^{-5s}[s]_{z^5}$, the sum of the $a$ and $b$ terms in
\eqref{eq:Xi} is
\[
 q^{Q(b)}z^{L(\rho b)}[s]_{z^5}
                   \bigl(\Theta_b-q^s\Theta_a\bigr).
\]
Every nonzero term is paired exactly once. Dividing
\eqref{eq:signed-divisibility} by the unit $D(q)^2$ proves
\eqref{eq:positive-lattice} as an identity.

It remains to prove the nonnegativity of $\mathcal B_b$.
Keep $a=Tb$ and $s=b_4>0$. Substitution into the charge map gives
\begin{equation}\label{eq:charge-change}
 c(b)-c(a)=(2s,s,0,-s,-2s).
\end{equation}
Suppose first that $s=2t$ is even. Then $t\ge1$, all five signs
and $\beta$ are unchanged, and $\Theta_b=\Theta_a$. Therefore
\begin{equation}\label{eq:even-positive}
 \mathcal B_b=\Theta_a\frac{1-q^{2t}}{D(q)^2}
   =\Theta_a\frac{1}{1-q^{2t}}
        \prod_{\substack{m\ge1\\m\ne t}}\frac1{(1-q^{2m})^2}.
\end{equation}
Every factor on the right is nonnegative.

Now suppose that $s=2t+1$. Exactly $\sigma_1$ and $\sigma_3$
change sign. Their values at $a$ agree: indeed,
\[
 c_1(a)+c_3(a)=-a_0+a_1-a_2+a_3
            \equiv1-a_4\equiv0\pmod2,
\]
and the fixed runner indices $1$ and $3$ have the same parity.
Write $\sigma_1(a)=\sigma_3(a)=\delta\in\{1,-1\}$.
Flipping these two signs changes their sum by $-4\delta$ and
therefore gives $\beta(b)=\beta(a)+\delta$.
The other three signs remain fixed. With
$S_\varepsilon(u,q)=S(u^\varepsilon,q)$ for
$\varepsilon\in\{1,-1\}$, it follows that
\begin{align}
 \mathcal B_b
 &=u^{\beta(a)}\prod_{r\in\{0,2,4\}}S_{\sigma_r(a)}(u,q)
       \frac{u^\delta S_{-\delta}(u,q)^2
                          -q^{2t+1}S_\delta(u,q)^2}{D(q)^2}
       \notag\\
 &=u^{\beta(a)}\prod_{r\in\{0,2,4\}}S_{\sigma_r(a)}(u,q)
          \mathcal F_{2,t}(u^\delta,q).
       \label{eq:odd-factor}
\end{align}
The second line uses $u\mapsto u^{-1}$ when $\delta=-1$.
Proposition~\ref{prop:kernel} makes the last factor nonnegative,
and all remaining factors are nonnegative Laurent series.
This proves the required positivity in both cases.

Finally, a contribution to a fixed degree $q^N$ must satisfy
$Q(b)\le N$. There are only finitely many such $b$ by
\eqref{eq:coordinate-bound}. For each of them, the kernel and theta
products have finite Laurent support in every degree, and
$[b_4]_{z^5}$ is a finite polynomial. This proves the asserted
local finiteness after pairing and after division by $D(q)^2$.
\end{proof}

\begin{proof}[Proof of Theorems~\ref{thm:main} and~\ref{thm:strong}]
Theorem~\ref{thm:expansion} and $\Omega=D^{10}\Psi$ give
\[
 D^8\frac{\Psi}{[5]_z}=\frac{\Omega}{[5]_zD^2}=\GG
       \in\Nzero[u,u^{-1},z,z^{-1}][[q]],
\]
which proves Theorem~\ref{thm:strong}. Multiplication by the
nonnegative series $D^{-8}$ gives
\begin{equation}\label{eq:H-generating}
 \frac{\Psi(u,z,q)}{[5]_z}
   =\frac{\GG(u,z,q)}{D(q)^8}
   =\sum_{n\ge0}q^{n+1}\sum_{j\in\Z}u^jH_{n,j}(z).
\end{equation}
Taking $[q^{n+1}u^j]$ proves Theorem~\ref{thm:main}.
All coefficient extractions are finite by the preceding theorem.
\end{proof}

\section{The one-factor boundary}\label{sec:one-factor}

\subsection{The local threshold}\label{subsec:one-factor-kernels}

We now determine when a single Euler denominator suffices for the
kernel $\mathcal F_{\ell,t}$ defined in \eqref{eq:general-kernel}.

\begin{proposition}\label{prop:kernel-exact-threshold}
For $\ell,t\in\Nzero$, one has
\[
 \mathcal F_{\ell,t}(u,q)\in\Nzero[u,u^{-1}][[q]]
 \quad\Longleftrightarrow\quad
 \ell\ge2\ \text{or}\ (\ell=1\text{ and }t\ge1).
\]
The negative coefficients of $\mathcal F_{1,0}$ are precisely
\begin{equation}\label{eq:one-factor-negative-locus}
 [u^Jq^{J^2-J+2}]\mathcal F_{1,0}=-1
 \qquad(J\text{ odd}).
\end{equation}
Every coefficient not listed in
\eqref{eq:one-factor-negative-locus} is nonnegative.
\end{proposition}

\begin{proof}
Write $E(x)=(x;x)_\infty$ and $x=q^2$. For every $d\ge2$,
\begin{equation}\label{eq:one-factor-large-square}
 \frac{(1-x^d)^2}{E(x)}
 =\frac{1+x+\cdots+x^{d-1}}
        {\displaystyle\prod_{\substack{k\ge1\\k\ne1,d}}(1-x^k)}
 \in\Nzero[[x]].
\end{equation}
Here the factors at the two distinct part sizes $1$ and $d$
are used separately.

The case $d=1$ has exactly one negative coefficient. To see this,
let $g(n)$ count partitions of $n$ with no part equal to $1$,
with $g(n)=0$ for $n<0$. Then
\[
 \frac{(1-x)^2}{E(x)}
  =(1-x)\prod_{k\ge2}(1-x^k)^{-1}
  =\sum_{n\ge0}\bigl(g(n)-g(n-1)\bigr)x^n.
\]
We have $g(0)=1$ and $g(1)=0$. For $n\ge3$, increasing the
largest part by one injects the partitions counted by $g(n-1)$
into those counted by $g(n)$. Its image has a unique largest
part, and decreasing that part recovers the source. For $n=2$
the source set is empty. The complements of this injection give
\begin{equation}\label{eq:exceptional-one-factor-square}
 \frac{(1-x)^2}{E(x)}
 =1-x+x^2+\sum_{d\ge2}
       \frac{x^{2d}}{\prod_{k=2}^{d}(1-x^k)}.
\end{equation}
The summation term records a partition with largest part $d$
appearing at least twice; the separate term $x^2$ records the
one-part partition $(2)$.

Apply the two telescopes
\eqref{eq:theta-difference-1}--\eqref{eq:theta-difference-2}
after division by $E(x)$. For even $J$, every square has
$d=2m\ge2$, so its contribution is nonnegative by
\eqref{eq:one-factor-large-square}. For odd $J$, the term $m=0$
is the exceptional square \eqref{eq:exceptional-one-factor-square};
all terms with $m\ge1$ have $d=2m+1\ge3$ and start at degree
$x^{2m^2}\ge x^2$. Thus at $t=0$ the sole negative coefficient
in each odd-$J$ coefficient series occurs at relative degree $x$,
and it equals $-1$. The prefactor $q^{J^2-J}$ gives
\eqref{eq:one-factor-negative-locus}.

For $t\ge1$, cancellation of the single Euler factor at size $t$
shows that
\[
 W_t(x):=\frac{1-x^t}{E(x)}
       =\prod_{\substack{k\ge1\\k\ne t}}(1-x^k)^{-1}
       \in\Nzero[[x]],\qquad W_t(0)=1.
\]
The correction in the even-$J$ case is $\vartheta_0W_t$,
which is nonnegative. In the odd-$J$ case it is
$x\vartheta_1W_t$, also nonnegative, and its coefficient at $x$
is $2$. It therefore removes the only negative coefficient just
identified. This proves positivity for $\ell=1,t\ge1$.
Proposition~\ref{prop:kernel} gives positivity for every $\ell\ge2$,
and \eqref{eq:one-factor-negative-locus} excludes $\ell=1,t=0$.

It remains to exclude $\ell=0$ for every fixed $t\ge0$.
At $J=0$, equation~\eqref{eq:kernel-numerator} gives
\[
 [u^0]\mathcal F_{0,t}=q\bigl(\vartheta_1(x)-x^t\vartheta_0(x)\bigr).
\]
For an integer $r\ge1$, the coefficient of $x^{t+2r^2}$
in $x^t\vartheta_0$ is $2$. It can receive a contribution from
$\vartheta_1$ only if, for some integer $s$,
\[
 t+2r^2=2s(s+1),\qquad
 (2s+1-2r)(2s+1+2r)=2t+1.
\]
The nonzero integer $2t+1$ has only finitely many ordered integer
factorizations. Thus only finitely many positive integers $r$
can satisfy this equation. For every sufficiently large $r$,
\[
 [u^0q^{1+2t+4r^2}]\mathcal F_{0,t}=-2.
\]
This rules out $\ell=0$ and completes the classification.
All sums introduced above are locally finite, since their starting
degrees are quadratic or at least linear in their indices.
\end{proof}

Multiplication by a mixed pair of theta factors makes the
one-factor kernel nonnegative also at $t=0$.

\begin{proposition}\label{prop:mixed-theta-kernel}
For every $t\in\Nzero$,
\begin{equation}\label{eq:mixed-theta-kernel}
 S(u,q)S(u^{-1},q)\mathcal F_{1,t}(u,q)
       \in\Nzero[u,u^{-1}][[q]].
\end{equation}
The same assertion holds with $u$ replaced by $u^{-1}$.
\end{proposition}

\begin{proof}
Set
\[
 \mathcal T(u,q)=\sum_{k\in\Z}u^kq^{k^2},\qquad
 \psi(x)=S(1,x),\qquad O(x)=(x;x^2)_\infty.
\]
We first verify
\begin{equation}\label{eq:mixed-theta-addition}
 S(u,q)S(u^{-1},q)=\psi(q^2)\mathcal T(u,q).
\end{equation}
In the coefficient of $u^k$ on the left, let the two theta indices
be $a,b$, so $a-b=k$, and put $w=a+b$. Then $w\equiv k\pmod2$
and the $q$ exponent is
\[
 a(2a-1)+b(2b-1)=k^2+w^2-w.
\]
If $k$ is even, write $w=2r$ to obtain $w^2-w=2r(2r-1)$;
if $k$ is odd, write $w=2r+1$ to obtain
$w^2-w=2r(2r+1)=2(-r)(2(-r)-1)$.
In either case the remaining sum is $\psi(q^2)$.

The two-core formula \eqref{eq:P-def}, specialized at
$u=1$ and with its variable $q$ replaced by $x$, gives
\[
 \frac1{E(x)}=\frac{\psi(x)}{E(x^2)^2},
 \qquad E(x)=(x;x)_\infty.
\]
Hence
\begin{equation}\label{eq:psi-euler-specialization}
 \frac{\psi(x)}{E(x)}
  =\frac{E(x^2)^2}{E(x)^2}=\frac1{O(x)^2}.
\end{equation}
Combining \eqref{eq:mixed-theta-addition} and
\eqref{eq:psi-euler-specialization}, with $x=q^2$, gives
\[
 S(u,q)S(u^{-1},q)\mathcal F_{1,t}(u,q)
 =\mathcal T(u,q)\frac{uS(u^{-1},q)^2-q^{2t+1}S(u,q)^2}{O(x)^2}.
\]
The series $\mathcal T$ is nonnegative. To prove the same for the remaining
quotient, apply \eqref{eq:kernel-numerator} and the two telescopes.
For every $d\ge1$,
\[
 \frac{(1-x^d)^2}{O(x)^2}
 =(1+x+\cdots+x^{d-1})^2
       \prod_{\substack{k\ge3\\k\text{ odd}}}(1-x^k)^{-2}
 \in\Nzero[[x]].
\]
For $t\ge1$, one likewise has
\[
 \frac{1-x^t}{O(x)^2}
 =\frac{1+x+\cdots+x^{t-1}}{1-x}
       \prod_{\substack{k\ge3\\k\text{ odd}}}(1-x^k)^{-2}
 \in\Nzero[[x]],
\]
and the correction is zero at $t=0$. Each summand of the
coefficient expansion is therefore nonnegative. Quadratic starting
degrees in the theta indices ensure local finiteness, including
after multiplication by $\mathcal T$. Finally, inversion of the Laurent
variable preserves coefficientwise nonnegativity, and the mixed
factor $S(u,q)S(u^{-1},q)$ is invariant under that inversion.
\end{proof}

\begin{corollary}\label{cor:mixed-five-theta}
Write
\[
 \Theta^{(k)}=u^{k/2-1}S(u,q)^{5-k}S(u^{-1},q)^k
 \qquad(k=0,2,4).
\]
Then both
\[
 \frac{\Theta^{(4)}-q\Theta^{(2)}}{D(q)},\qquad
 \frac{\Theta^{(2)}-q\Theta^{(4)}}{D(q)}
\]
have nonnegative Laurent coefficients.
\end{corollary}

\begin{proof}
With $S_+=S(u,q)$ and $S_-=S(u^{-1},q)$, the two quotients are,
respectively,
\[
 S_-\bigl(S_+S_-\mathcal F_{1,0}(u,q)\bigr),\qquad
 uS_-\bigl(S_+S_-\mathcal F_{1,0}(u^{-1},q)\bigr).
\]
Apply Proposition~\ref{prop:mixed-theta-kernel}.
\end{proof}

\subsection{The remaining exceptional lattice terms}

Only two parity classes remain in the paired expansion. Define
\[
 \mathcal E=\{b\in\A:b_4=1,\ 
                  \{k(b),k(Tb)\}=\{0,2\}\}.
\]

\begin{corollary}\label{cor:exceptional-pairs}
For every $b\in\A$ with $b_4>0$ and $b\notin\mathcal E$,
\[
 \frac{\Theta_b-q^{b_4}\Theta_{Tb}}{D(q)}
       \in\Nzero[u,u^{-1}][[q]].
\]
Consequently all summands outside $\mathcal E$ in the paired
expansion of $\Omega/([5]_zD)$ are nonnegative. A vector belongs
to $\mathcal E$ exactly when $b_4=1$ and either all five coordinates
are odd, or $b_0,b_3,b_4$ are odd while $b_1,b_2$ are even.
\end{corollary}

\begin{proof}
Write $s=b_4$. If $s$ is even, the paired difference is
$\Theta_{Tb}(1-q^s)$; its quotient by $D$ is nonnegative after
canceling the Euler factor at size $s/2$.
If $s=2t+1\ge3$, the factorization \eqref{eq:odd-factor}, with
one Euler denominator in place of two, contains
$\mathcal F_{1,t}(u^\delta,q)$. This is nonnegative by
Proposition~\ref{prop:kernel-exact-threshold}.
When $s=1$, the two changed signs agree, so the pair of sign counts
is either $\{0,2\}$ or $\{2,4\}$. Corollary~\ref{cor:mixed-five-theta}
handles both orientations of the latter pair.

Finally, $k(b)=0$ means that every two adjacent coordinates have
the same parity. Since $b_4=1$, this is equivalent to all coordinates
being odd. Similarly, $k(Tb)=0$ means that
$b_0,b_1+1,b_2+1,b_3,-1$ are all odd, which gives the second
parity pattern. In each case the other sign count is two.
\end{proof}

The exceptional terms cannot all be made nonnegative individually.
For example, $b=(1,-1,-1,1,1)$ has $Q(b)=4$, $k(b)=0$, and
$k(Tb)=2$. Directly from
$S(u,q)=1+uq+u^{-1}q^3+O(q^6)$ and $D(q)^{-1}=1+q^2+O(q^4)$,
\[
 [u^{-1}q^2]\frac{\Theta^{(0)}-q\Theta^{(2)}}{D(q)}=-1.
\]
Here the first numerator term has no $u^{-1}q^2$ term, the
second contributes $-2$, and the Euler denominator contributes
$1$ from the constant $u^{-1}$ term.
Thus this term requires compensation from a different lattice term,
as constructed in the next section.

\section{The sharp Euler-factor threshold}
\label{sec:nine-factors}

To remove a ninth Euler factor, we group paired terms indexed by
different vectors $b$ with $b_4>0$. Define
\[
 \mathsf A=\frac{\Theta^{(0)}-q\Theta^{(2)}}D,\qquad
 \mathsf A'=\frac{\Theta^{(2)}-q\Theta^{(0)}}D,
\]
\[
 \mathsf B=\frac{\Theta^{(2)}-q\Theta^{(4)}}D,\quad
 \mathsf C=\frac{\Theta^{(4)}-q\Theta^{(2)}}D,\quad
 \mathsf E=\frac{(1-q^2)\Theta^{(2)}}D.
\]
Here $\Theta_b=\Theta^{(k(b))}$; the five symbols
$\mathsf A,\mathsf A',\mathsf B,\mathsf C,\mathsf E$ denote the lattice kernels above.
The one-factor results already prove $\mathsf B,\mathsf C,\mathsf E\ge0$.
The additional local inputs are
\begin{equation}\label{eq:three-local-compensations}
 \mathsf B+q^3\mathsf A\ge0,\qquad
 \mathsf C+q^3\mathsf A'\ge0,\qquad
 \mathsf A'+2q^6\mathsf E\ge0.
\end{equation}
The first two are proved in Propositions~\ref{prop:local-L1}
and~\ref{prop:local-L2}; the third is
Proposition~\ref{prop:two-slot-compensation}. Each inequality is coefficientwise in $u$ and $q$.

The following lemmas describe the coefficient allocations and
regroupings used below.

\begin{lemma}[Coefficient allocation]\label{lem:coefficient-allocation}
Fix $r\in\Nzero$. Let $\mathcal P$ and $\mathcal N$ be sets of
positive and negative coefficient labels, respectively, and give each
label $a$ a weight $(d(a),\mathbf v(a))\in\Nzero\times\Z^r$.
Assume that only finitely many labels have $d(a)\le N$ for each
$N\in\Nzero$. If an injection $\iota:\mathcal N\longrightarrow\mathcal P$
preserves the weight, then
\[
 \sum_{p\in\mathcal P}q^{d(p)}\mathbf z^{\mathbf v(p)}
 -\sum_{n\in\mathcal N}q^{d(n)}\mathbf z^{\mathbf v(n)}
 \in\Nzero[z_1^{\pm1},\ldots,z_r^{\pm1}][[q]].
\]
Here $\mathbf z^{\mathbf v}=\prod_{i=1}^r z_i^{v_i}$, with
empty product $1$ when $r=0$. A coefficient of magnitude $c$
is represented by $c$ distinct labels.
The number of positive labels at a weight is its \emph{capacity}.
Sources of the same weight must be assigned different labels.
\end{lemma}
\begin{proof}
For each full weight, the restriction of $\iota$ is an injection
between finite sets. Subtract their cardinalities. The assumed
finiteness makes these comparisons the coefficients of the displayed
formal series.
\end{proof}

The local inequalities also hold with the shifts required by the
lattice allocation. We first prove the needed difference estimates.

\begin{lemma}\label{lem:compensation-differences}
For every integer $h\ge1$ and $r=1,2$,
\[
 (1-q^{10h})^r\mathsf E\ge0,\qquad
 (1-q^{10h})^r\mathsf B\ge0,\qquad
 (1-q^{10h})^r\mathsf C\ge0.
\]
\end{lemma}

\begin{proof}
Put $x=q^2$ and $d=5h\ge5$. For $\mathsf E$, the first
power cancels the distinct Euler factors at sizes one and $d$.
For the second power it remains to prove
\[
 (1-x^d)\prod_{k\notin\{1,d\}}(1-x^k)^{-1}\ge0.
\]
A nonempty partition with neither a part one nor a part $d$ maps
injectively to one whose size is $d$ larger, by increasing its
largest part by $d$.
The new largest part is unique and at least $d+2$, so the map is
reversible on its image. For the empty source use $(d-2,2)$;
its parts are allowed and its largest part is less than $d+2$.
Thus this additional image is disjoint from every nonempty-source
image, proving the inequality at every degree.

For $\mathsf B$ and $\mathsf C$, the mixed-theta proof at shift
zero expresses each as a nonnegative series times
\[
 R(x)=\prod_{\substack{k\ge3\\ k\text{ odd}}}(1-x^k)^{-2}.
\]
Indeed each square in the telescopes becomes a nonnegative finite
geometric square times this same product. Its two factors at size
five absorb $(1-x^{5h})^r$ for $r\le2$, since
$(1-x^{5h})/(1-x^5)=1+x^5+\cdots+x^{5(h-1)}$.
\end{proof}

In particular, the first two inequalities in
\eqref{eq:three-local-compensations} imply, for every $h\ge0$,
\begin{equation}\label{eq:shifted-local-compensations}
 \mathsf B+q^{3+10h}\mathsf A\ge0,
 \qquad \mathsf C+q^{3+10h}\mathsf A'\ge0.
\end{equation}
For example, the left side of the first inequality is
$q^{10h}(\mathsf B+q^3\mathsf A)+(1-q^{10h})\mathsf B$.

\begin{proof}[Proof of Theorem~\ref{thm:nine-factors}]
By the paired identity, after expanding the finite geometric factor,
\begin{equation}\label{eq:nine-slot-expansion}
 \frac\Omega{[5]_zD}
 =\sum_{\substack{b\in\A\\s=b_4>0}}
   \sum_{v=0}^{s-1}q^{Q(b)}z^{Z(b)+5v}
               \frac{\Theta_b-q^s\Theta_{Tb}}D,
 \qquad Z(b)=L(\rho b).
\end{equation}
We call an index $(b,v)$ a \emph{lattice slot}; it denotes one
occurrence of the paired kernel in this sum.
Corollary~\ref{cor:exceptional-pairs} leaves only the slots with
$s=1$ and parity pattern $11111$ or $10011$.
The five bits give the parities of $b_0,\ldots,b_4$, with one
meaning odd. Their kernels are $\mathsf A$ and $\mathsf A'$,
respectively. We give each exceptional slot one or two
nonexceptional target slots with the same integer $z$ exponent.

Use the displacements
\[
 w_0=(0,-2,0,1,1),\quad w_1=(1,0,-2,0,1),
 \quad v_0=(1,-1,-1,1,0).
\]
For $b_4=1$, direct calculation gives
\begin{align*}
 Q(b+w_0)-Q(b)&=6-5b_1,& Z(b+w_0)&=Z(b),\\
 Q(b+w_1)-Q(b)&=6-5b_2,& Z(b+w_1)&=Z(b)-5,\\
 Q(b-v_0)-Q(b)&=7-5(b_0+b_3),&Z(b-v_0)&=Z(b).
\end{align*}
The first two targets have last coordinate two and use slots zero
and one, respectively; the third has last coordinate one. Thus all
three assignments preserve the $z$ monomial in
\eqref{eq:nine-slot-expansion}. The target vectors belong to $\A$,
so their $Q$-values are at least one.

\smallskip\noindent
\textbf{Case 1: all coordinates of $b$ are odd.}
Use the following rule in the stated order:
if $b_1\ge1$, assign $(b+w_0,0)$; otherwise, if $b_2\ge1$,
assign $(b+w_1,1)$; if both $b_1,b_2\le-1$, assign $(b-v_0,0)$.
In either of the first two cases write the selected coordinate
as $2h+1$, $h\ge0$. The target kernel is $\mathsf E$, and the
combined contribution, apart from its common $z$ monomial, is
\[
 q^{Q(b)}\mathsf A+q^{Q(b)+1-10h}\mathsf E
 =q^{Q(b)}\frac{\Theta^{(0)}-q^3\Theta^{(2)}}D
   +q^{Q(b)+1-10h}(1-q^{10h})\mathsf E.
\]
The first term is $q^{Q(b)}S(u,q)^3\mathcal F_{1,1}(u^{-1},q)\ge0$;
the second is nonnegative by Lemma~\ref{lem:compensation-differences}
(and is zero at $h=0$).
In the third case write $b_0+b_3=-b_1-b_2=2h+2$, $h\ge0$.
The target has kernel $\mathsf B$ and lies $3+10h$ lower in
$Q$, so positivity follows from the first inequality in
\eqref{eq:shifted-local-compensations}.

For example, the exceptional vector $b=(1,-1,-1,1,1)$ considered
above has $Q(b)=4$ and $Z(b)=0$. The rule assigns
$c=b-v_0=(0,0,0,0,1)$, with $Q(c)=1$, $Z(c)=0$, and kernel
$\mathsf B$. Both slots have index zero, so their group is
\[
 q^4\mathsf A+q\mathsf B=q(\mathsf B+q^3\mathsf A)\ge0.
\]

\smallskip\noindent
\textbf{Case 2: $b$ has parity $10011$.}
Write $b_1=2p$ and $b_2=2r$.
If $p+r\le-1$, assign $(b-v_0,0)$. Its kernel is $\mathsf C$;
writing $h=-p-r-1\ge0$, it lies $3+10h$ lower in $Q$.
The second inequality in \eqref{eq:shifted-local-compensations}
proves positivity of the group.
If $p+r\ge0$, assign both $(b+w_0,0)$ and $(b+w_1,1)$.
Their kernels are $\mathsf E$. For this comparison we temporarily
work in $\Z[u,u^{-1}]((q))$, the ring of $q$-Laurent series with
exponents bounded below. The group is
\[
 q^{Q(b)}\bigl(\mathsf A'
       +q^6(q^{-10p}+q^{-10r})\mathsf E\bigr).
\]
The third inequality in \eqref{eq:three-local-compensations}
reduces this case to
\[
 (q^{-10p}+q^{-10r}-2)\mathsf E\ge0.
\]
If $p,r\ge0$, each difference $q^{-10p}-1$ is a monomial
multiple of the first difference in
Lemma~\ref{lem:compensation-differences}; for $p=0$ it is zero,
and the same holds for $r$. Otherwise interchange
$p,r$ and write $p=m\ge n\ge1$, $r=-n$. Then
\[
 (q^{-10m}+q^{10n}-2)\mathsf E
 =q^{-10m}(1-q^{10(m-n)})\mathsf E
   +q^{-10n}(1-q^{10n})^2\mathsf E\ge0.
\]
When $m=n$, the first term is zero.
Multiplying the comparison by $q^{Q(b)+6}$ returns it to
$\Z[u,u^{-1}][[q]]$. If $p,r\ge0$, the starting exponents
$Q(b)+6-10p$ and $Q(b)+6-10r$ are the two target values of $Q$.
In the mixed-sign case, $Q(b)+6-10m$ is one target value of $Q$,
and $Q(b)+6-10n$ is no smaller. All these exponents are at least one.

For example, $b=(1,2,-2,-1,1)$ has $p=1$, $r=-1$,
$Q(b)=11$, and $Z(b)=-5$. The targets $(b+w_0,0)$ and
$(b+w_1,1)$ have $(Q,Z,v)=(7,-5,0)$ and $(27,-10,1)$,
respectively, so their common $z$-weight is $z^{-5}$. Their group is
\[
 \begin{aligned}
 &z^{-5}(q^{11}\mathsf A'+q^7\mathsf E+q^{27}\mathsf E)\\
 &\qquad=z^{-5}\!\left[q^{11}(\mathsf A'+2q^6\mathsf E)
                 +q^7(1-q^{10})^2\mathsf E\right]\ge0.
 \end{aligned}
\]
Thus the displayed negative shift has starting degree seven after
restoration.

\smallskip\noindent
\textbf{Disjointness of the assignments.}
Within each case the assignment is a translation and hence injective.
The following table excludes coincidences between different cases.
Here $s'$ is the last coordinate of the target vector and $v'$ is
its slot index:
\[
\begin{array}{c|c|c|c|c}
 \text{source case}&s'&v'&\text{target parity}&\text{kernel}\\\hline
 11111,\ b_1\ge1&2&0&11100&\mathsf E\\
 11111,\ b_1\le-1,b_2\ge1&2&1&01110&\mathsf E\\
 11111,\ b_1,b_2\le-1&1&0&00001&\mathsf B\\
 10011,\ p+r\le-1&1&0&01101&\mathsf C\\
 10011,\ p+r\ge0\ (w_0)&2&0&10000&\mathsf E\\
 10011,\ p+r\ge0\ (w_1)&2&1&00010&\mathsf E
\end{array}
\]
Different $v'$ values denote different slots even if their target
vectors were equal. The parity entries with the same $s',v'$
are distinct. All target kernels are among the nonexceptional
kernels, so no exceptional source is itself consumed as a target.
The priority rule makes the source cases exhaustive and disjoint.

Group \eqref{eq:nine-slot-expansion} by these finite disjoint
assignments and leave every unused nonexceptional slot on its own.
Every group is nonnegative. For each group $G$, put
\[
 \nu(G)=\min\{Q(c):(c,v)\text{ is an original slot in }G\}.
\]
Each slot kernel has no negative $q$ exponents, so a group
contributing in degrees at most $N$ must have $\nu(G)\le N$.
Mark one slot attaining this minimum in each such group.
The marked slots are distinct because the groups are disjoint.
By \eqref{eq:coordinate-bound}, only finitely many vectors have
$Q(c)\le N$, and each has finitely many slots. Hence only finitely
many groups contribute through degree $N$. Regrouping is legitimate
coefficient by coefficient and proves $\Omega/([5]_zD)\ge0$.
Multiplication by $D^{-(9-M)}$ proves the assertion for $0\le M\le9$.
The reverse implication is Proposition~\ref{prop:global-obstruction}.
\end{proof}

\subsection{An obstruction to complete Euler-factor removal}

The following coefficient excludes removal of ten or more Euler factors.

\begin{proposition}\label{prop:global-obstruction}
One has
\begin{equation}\label{eq:global-obstruction}
 [q^7u^0z^0]\frac{\Omega(u,z,q)}{[5]_z}=-4.
\end{equation}
Consequently, if $M\in\Nzero$ and
$D(q)^M\Psi(u,z,q)/[5]_z$ has nonnegative coefficients, then $M\le9$.
\end{proposition}

\begin{proof}
Before division by $D^2$, the paired identity in the proof of
Theorem~\ref{thm:expansion} reads
\begin{equation}\label{eq:bare-pairing}
 \frac{\Omega}{[5]_z}
 =\sum_{\substack{b\in\A\\s=b_4>0}}
 q^{Q(b)}\sum_{v=0}^{s-1}z^{L(\rho b)+5v}
                  (\Theta_b-q^s\Theta_{Tb}).
\end{equation}
To contribute to $q^7z^0$, one must have $Q(b)\le7$ and
$L(\rho b)+5v=0$. The coordinate bound
\eqref{eq:coordinate-bound} gives
$s\le1/5+2\sqrt{7/5}<3$, so $s=1$ or $2$.
Writing $y=b_2$ and $w=b_3$, the coordinate sum and the $z$-condition
force
\[
 b=(2-2s+5v+y+2w,\ s-1-5v-2y-3w,\ y,\ w,\ s).
\]
There are three possible pairs $(s,v)$. Direct completion of squares
gives, respectively,
\begin{align*}
 40Q(b)&=(20y+25w-1)^2+175w^2-70w+39,
     &&(s,v)=(1,0),\\
 40Q(b)&=(20y+25w-16)^2+175w^2-320w+344,
     &&(s,v)=(2,0),\\
 40Q(b)&=(20y+25w+34)^2+175w^2+180w+244,
     &&(s,v)=(2,1).
\end{align*}
In the first case $Q\le7$ forces $w=0$ or $1$, and then
$(y,w)=(0,0)$ or $(-1,1)$. In the second case it forces $w=1$
and then $y=0$. In the third case $w=-1$ or $0$; when $w=-1$
the inequality $(20y+9)^2\le41$ has no integer solution, and when
$w=0$ it forces $y=-2$. Thus the following list is exhaustive:
\[
\begin{array}{c|c|c|c|c}
 b&v&Q(b)&k(b)&k(Tb)\\\hline
 (0,0,0,0,1)&0&1&2&4\\
 (1,-1,-1,1,1)&0&4&0&2\\
 (0,-2,0,1,2)&0&7&2&2\\
 (1,0,-2,0,2)&1&7&2&2
\end{array}
\]
Let
$\Theta^{(k)}=u^{k/2-1}S(u,q)^{5-k}S(u^{-1},q)^k$.
Using
\[
 S(u,q)=1+uq+u^{-1}q^3+u^2q^6+O(q^{10})
\]
gives
\begin{align*}
 [q^6u^0]\Theta^{(2)}&=12,&
 [q^5u^0]\Theta^{(4)}&=12,\\
 [q^3u^0]\Theta^{(0)}&=0,&
 [q^2u^0]\Theta^{(2)}&=6,&
 [q^0u^0]\Theta^{(2)}&=1.
\end{align*}
The first value $12$ is the sum of six choices of one
$q^3$ term from each sign and six choices of one $q^3$ term and
three $q$ terms. The second value $12$ chooses one of four $q^3$
terms in $S(u^{-1},q)^4$ and two of the other three $q$ terms.
The coefficient $6$ chooses one $q$ term from each sign, in
$3\cdot2$ ways. These account for all possible degrees.
The four rows of the table contribute
\[
 (12-12)+(0-6)+1+1=-4
\]
to \eqref{eq:bare-pairing}, proving \eqref{eq:global-obstruction}.

If $M\ge10$ and $D^M\Psi/[5]_z$ were nonnegative, multiplication
by the nonnegative series $D^{-(M-10)}$ would make
$D^{10}\Psi/[5]_z=\Omega/[5]_z$ nonnegative, contradicting the
coefficient just calculated.
\end{proof}

\begin{remark}\label{rem:nine-unimodality}
The numerator of a slice in the nine-factor theorem need not be
unimodal. For example, put
\[
 P_7(z)=[q^7u^0]\frac{\Omega(u,z,q)}{D(q)}.
\]
Writing $Z_r=z^r+z^{-r}$, one has
\[
 P_7(z)=[5]_z\bigl(11+27Z_1+18Z_2+24Z_3+6Z_4+5Z_5+4Z_6\bigr).
\]
This follows from \eqref{eq:Theta-Omega} by using
$S(u,q)=1+uq+u^{-1}q^3+u^2q^6+O(q^{10})$ and
$D^{-1}=1+q^2+2q^4+3q^6+O(q^8)$.
Only $Q(a)\le7$ contributes, and \eqref{eq:coordinate-bound}
restricts every coordinate to $-2\le a_i\le2$.
The coefficients of $z^{-1},1,z$ in the symmetric Laurent polynomial
$z^{-2}P_7(z)$ are $107,101,107$. Thus the numerator is not unimodal,
although the displayed quotient is nonnegative.
The criterion for symmetric, unimodal numerators
of~\cite[Lemma~3.1]{BGRT} therefore does not apply directly to this slice.
\end{remark}

\section{Theta identities for the local inequalities}\label{sec:common-theta}

The three local inequalities use the same quartic slices and
transformations of the BG index. The identities are proved directly from the theta series, so they
apply to the signed kernels that occur in the local inequalities.

\begin{lemma}[Common theta-index transformations]
\label{lem:common-theta-index-transformations}
For $d\in\{4,5\}$ and $k\in\{0,2,4\}$ with $k\le d$, put
\[
 U_{d,k}(u,q)=u^{k/2-1}S(u,q)^{d-k}S(u^{-1},q)^k,
 \qquad f(t)=t(2t-1).
\]
Write a theta term as an integer tuple $(j_1,\ldots,j_d)$, with
$k$ signs $\sigma_i=-1$, the remaining signs $\sigma_i=1$, and
\[
 J=k/2-1+\sum_i\sigma_i j_i,
 \qquad N=\sum_i f(j_i).
\]
For every $h\in\Z$, the following bijections change $(J,N)$
as indicated:
\[
\begin{array}{c|c|c}
 \text{index map}&\text{new }J&N'-N\\\hline
 j_i\mapsto(1+\sigma_i)/2-j_i&d-2-J&d-2-2J\\
 j_i\mapsto(1-\sigma_i)/2-j_i&-2-J&2J+2\\
 j_i\mapsto j_i+h\sigma_i&J+dh&2dh^2+(4J+4-d)h
\end{array}
\]
In particular, the degree changes do not depend on $k$. These
relations therefore apply to any finite linear combination
$\sum_k c_k(q)U_{d,k}(u,q)$ with $c_k(q)\in\Z[[q]]$. Relations with a negative monomial
shift are interpreted after extracting $u^J$ in $\Z((q))$, the
ring of Laurent series whose exponents are bounded below.
Equivalently, move the negative power to the other side.
For $d=5$, coefficientwise nonnegativity at $J=-1,0,1$ implies
coefficientwise nonnegativity at every integer $J$.
\end{lemma}
\begin{proof}
For $\sigma\in\{-1,1\}$, direct expansion gives
\[
 f\bigl((1+\sigma)/2-j\bigr)-f(j)
   =(1+\sigma)/2-2\sigma j,
\]
\[
 f\bigl((1-\sigma)/2-j\bigr)-f(j)
   =(1-\sigma)/2+2\sigma j,
\]
\[
 f(j+h\sigma)-f(j)=4h\sigma j+2h^2-h\sigma.
\]
Sum over $i$, using $\sum_i\sigma_i=d-2k$ and
$\sum_i\sigma_i j_i=J-k/2+1$. The first two maps are involutions,
and the third has inverse obtained by replacing $h$ with $-h$.
The change of total degree is common to all $k$, proving the
assertion for scalar linear combinations, even when their coefficients
are signed. A coefficient in $c_k(q)$ is left unchanged by each
index map, so the same shift also applies to every term of the
product. When $d=5$, translation
reduces $J$ to one of $-1,0,1,2,3$. The first reflection sends
$2$ to $1$ and $3$ to $0$. Each relation is a bijection with a
fixed monomial degree shift, and therefore preserves coefficient
signs in both directions. All relevant theta sums are locally
finite because $f(j)\ge0$ and tends to infinity with $|j|$.
\end{proof}

\begin{lemma}[Classical quartic slices]\label{lem:quartic-slices}
Put $X=S(u,q)$, $Y=S(u^{-1},q)$, and $x=q^2$. Define
\[
 \phi(x)=\sum_{r\in\Z}x^{r^2},\qquad
 \psi(x)=\sum_{r\ge0}x^{r(r+1)/2}
        =\sum_{j\in\Z}x^{j(2j-1)},\qquad E(x)=(x;x)_\infty.
\]
The elementary addition, duplication, and two-core identities are
\begin{equation}\label{eq:l1-classical-identities}
 \psi(x)^2=\phi(x)\psi(x^2),\quad
 \phi(x)^2=\phi(x^2)^2+4x\psi(x^4)^2,\quad
 \psi(x)=\frac{(x^2;x^2)_\infty^2}{(x;x)_\infty}.
\end{equation}
For $P=\phi(x)$, $M=\phi(-x)$, $T_0=\psi(x^2)$ and
$U_0=u^{-1}X^4$, $U_2=X^2Y^2$, $U_4=uY^4$, one has
\begin{equation}\label{eq:l1-hadamard-table}
\begin{array}{c|ccc}
 s &[u^s]U_0 &[u^s]U_2 &[u^s]U_4\\ \hline
 0&4qT_0^3&P^2T_0&4qT_0^3\\
 1&(P^3-M^3)/2&2qPT_0^2&(P^3+M^3)/2\\
 -1&(P^3+M^3)/2&2qPT_0^2&(P^3-M^3)/2.
\end{array}
\end{equation}
More generally,
\begin{equation}\label{eq:quartic-fixed-sum}
 [u^m]X^4=
 \begin{cases}
 4q^{(m^2-2m+3)/2}\psi(q^4)^3,&m\text{ odd},\\
 q^{2r^2-2r}\bigl(P^3+(-1)^rM^3\bigr)/2,&m=2r.
 \end{cases}
\end{equation}
\end{lemma}

\begin{proof}
The two sums defining $\psi$ agree: positive $j$ gives
$r=2j-1$, whereas nonpositive $j$ gives $r=-2j$. The first two
identities in \eqref{eq:l1-classical-identities} follow by separating
parity classes after $(r,s)\mapsto(r+s,r-s)$. The last identity is the two-core specialization of the
core--quotient formula; see \eqref{eq:P-def} and~\cite{Andrews,CJW}.

For the quartic formula, fix the sum $m$ of the indices
$j_1,j_2,j_3,j_4$ in $X^4$ and set
\[
 p=j_1+j_2-j_3-j_4,\quad
 r'=j_1-j_2+j_3-j_4,\quad
 t=j_1-j_2-j_3+j_4.
\]
The inverse map is
\[
 \begin{aligned}
  j_1&=(m+p+r'+t)/4,& j_2&=(m+p-r'-t)/4,\\
  j_3&=(m-p+r'-t)/4,& j_4&=(m-p-r'+t)/4.
 \end{aligned}
\]
The four transformed coordinates have the same parity. Under this
condition the inverse is integral if and only if
$m+p+r'+t\equiv0\pmod4$. The change of variables gives
\[
 4\sum_i j_i^2=m^2+p^2+(r')^2+t^2,
 \qquad
 \sum_i j_i(2j_i-1)=\frac{m^2+p^2+(r')^2+t^2}{2}-m.
\]
If $m=2r$, write $(p,r',t)=2(i,j,k)$. The inverse coordinates
are integers precisely when $i+j+k\equiv r\pmod2$.
The remaining theta sum is therefore $(P^3+(-1)^rM^3)/2$,
and the scalar factor is $q^{2r^2-2r}$.
If $m$ is odd, the three transformed coordinates are odd and the
inverse condition is $p+r'+t\equiv-m\pmod4$. This selects one
of the two possible odd residue classes modulo four. Negating $p$ exchanges these classes,
has no fixed points, and preserves the quadratic weight. For an odd coordinate $p$,
$\sum_{p\text{ odd}}q^{(p^2-1)/2}=2\psi(q^4)$.
After extracting $q^{(m^2-2m+3)/2}$ from the norm formula, taking
half of the unrestricted product proves the odd case of
\eqref{eq:quartic-fixed-sum}. All exponents in this calculation
are integers.

Finally, \eqref{eq:mixed-theta-addition} gives
$XY=\psi(q^2)\sum_hu^hq^{h^2}$. Squaring and using
$\psi(x)^2=P T_0$ yields the middle column of
\eqref{eq:l1-hadamard-table}. Applying
\eqref{eq:quartic-fixed-sum} with $m=s+1$ for $U_0$ and $m=1-s$
for $U_4$ yields the other columns. All sums are locally finite,
so the index changes are identities of formal series.
\end{proof}

In the three local proofs, $x=q^2$ and $E=E(x)=D(q)$; this
Euler product is distinct from the lattice kernel $\mathsf E$.
Other auxiliary scalar notation is local to the section or lemma
where it is defined. In particular, $A_0,B_0$ in
Section~\ref{sec:local-L1} are distinct from $\mathsf A,\mathsf B$.
We shall use the following common scalar notation:
\[
 \begin{gathered}
 a=\phi(x^2),\quad b=2\psi(x^4),\quad K=a-xb,\quad L=b-a,\\
 V=K/E(x)+x,\quad G=L/E(x),\quad
 Q_3(x)=\prod_{n\ge3}(1-x^n)^{-1}.
 \end{gathered}
\]
For a scalar series $H(x)\in\Z[[x]]$, write $H_n=[x^n]H(x)$.

\begin{lemma}\label{lem:common-scalar-bounds}
One has
\begin{equation}\label{eq:compensated-basic-bounds}
 V\ge0,\qquad V_n\ge1\ (n\ge2),\qquad G\ge(1+x)Q_3.
\end{equation}
The small coefficients needed below are
\begin{equation}\label{eq:l2-basic-bounds}
 V_0=1,\qquad V_1=0,\qquad V_3=1,
\end{equation}
\begin{equation}\label{eq:l2-g-small}
 G_0=G_1=G_3=1,\quad G_2=0,\quad
 G_4=G_5=3,\quad G_6=5,\quad G_7=7.
\end{equation}
Moreover,
\begin{equation}\label{eq:compensated-partition-lower-bound}
 (Q_3)_n\ge\lfloor n/2\rfloor-1\qquad(n\ge6).
\end{equation}
\end{lemma}
\begin{proof}
Since $a=\vartheta_0(x)$ and $b=\vartheta_1(x)$, the telescopes
\eqref{eq:theta-difference-1}--\eqref{eq:theta-difference-2} give
\[
 K=\sum_{m\ge0}x^{2m^2}(1-x^{2m+1})^2,
 \qquad
 L=\sum_{m\ge1}x^{2m(m-1)}(1-x^{2m})^2.
\]
By \eqref{eq:one-factor-large-square} and
\eqref{eq:exceptional-one-factor-square}, the sole negative
coefficient of $K/E$ is $-1$ at $x$, so $V\ge0$.
The $m=1$ summand of $K/E$ is
\[
 x^2\frac{(1-x^3)^2}{E}
 =x^2(1+x+x^2)\prod_{\substack{n\ge1\\ n\ne1,3}}(1-x^n)^{-1}.
\]
Replacing $1+x+x^2$ by $1+x$ and retaining only the size-two
Euler factor gives the coefficientwise lower bound
$x^2(1+x)/(1-x^2)=x^2/(1-x)$. Thus this summand supplies at
least one coefficient at every degree from two onwards. The $m=1$ summand
of $L/E$ is $(1+x)Q_3$, giving the bound for $G$.
The small coefficients follow from
$K=1-2x+2x^2+O(x^5)$ and $L=1-2x^2+2x^4+O(x^8)$.
Finally, $Q_3$ counts partitions with all parts at least three.
For $n\ge6$, counting its one-part partition and its two-part
partitions gives
\begin{equation}\label{eq:l2-q3-bound}
 (Q_3)_n\ge1+\bigl(\lfloor n/2\rfloor-2\bigr)
          =\lfloor n/2\rfloor-1.
\end{equation}
\end{proof}

\section{A sharp compensated one-factor kernel}
\label{subsec:compensated-one-factor}

The difference $\Theta^{(2)}-q\Theta^{(0)}$ can be compensated
by two even-pair contributions. Two is the least possible coefficient
in the family below.

\begin{proposition}\label{prop:two-slot-compensation}
For every integer $c\ge0$, put
\[
 \mathcal C_c(u,q)=
 \frac{\Theta^{(2)}-q\Theta^{(0)}
       +c q^6(1-q^2)\Theta^{(2)}}{D(q)}.
\]
Then $\mathcal C_c$ has nonnegative Laurent coefficients if and only
if $c\ge2$.
\end{proposition}

\begin{proof}
Write $X=S(u,q)$, $Y=S(u^{-1},q)$, and $x=q^2$.
Since $(1-q^2)\Theta^{(2)}/D\ge0$, it suffices to prove
$\mathcal C_2\ge0$ and to exclude $c<2$.
Factor $\mathcal C_2=X\mathcal R$, where
\[
 \mathcal R=
 \frac{(1+2q^6-2q^8)X^2Y^2-q u^{-1}X^4}{D}.
\]
We first determine the negative coefficients of $\mathcal R$, and
then give an injection that compensates them after multiplication
by $X$.

Retain the notation of Lemma~\ref{lem:common-scalar-bounds}, and put
$P=\phi(x)$, $M=\phi(-x)$, and $T=\psi(x^2)$. The classical quartic slice table \eqref{eq:l1-hadamard-table} in
Lemma~\ref{lem:quartic-slices}, together with the index transformations
of Lemma~\ref{lem:common-theta-index-transformations}, gives
\begin{equation}\label{eq:compensated-quartic-slices}
 [u^s]\mathcal R=
 \begin{cases}
 q^{s^2/2}f_e(x),&s\text{ even},\\
 q^{(s^2+1)/2}f_+(x),&s\equiv1\pmod4,\\
 q^{(s^2+1)/2}f_-(x),&s\equiv-1\pmod4,
 \end{cases}
\end{equation}
where
\begin{align}
 f_e&=\frac{T\bigl((1+2x^3-2x^4)P^2-4xT^2\bigr)}{E},
       \label{eq:compensated-fe}\\
 f_\pm&=\frac{2(1+2x^3-2x^4)PT^2
                 -(P^3\mp M^3)/2}{E}.
       \label{eq:compensated-fpm}
\end{align}
We use the scalar bounds in Lemma~\ref{lem:common-scalar-bounds}.

The classical identities $P^2=a^2+xb^2$ and $T^2=ab/2$ yield
an explicit positive decomposition of the even slice. Define
\[
 Z=bG-a+2x^2,\qquad W=(1-x)P^2/E.
\]
Then $Z\ge0$, $Z_0=1$, and $W\ge1$. Indeed, the only subtracted
coefficients of $a$ are one at degree zero and two at degrees
$2r^2$, $r\ge1$. Degree zero is covered by $b_0G_0=2$;
degree two is repaired by $2x^2$; and at degrees $2r^2\ge8$,
\eqref{eq:compensated-basic-bounds} gives
$bG\ge2(1+x)\prod_{k\ge3}(1-x^k)^{-1}$, whose coefficient is
at least four. Canceling the size-one Euler factor gives $W\ge1$. Direct expansion now gives
\begin{equation}\label{eq:compensated-even-positive}
 f_e=T\bigl(aV+xZ+2x^3(W-1)\bigr)\ge0.
\end{equation}
In particular,
\begin{equation}\label{eq:compensated-capacity}
 [x]f_e=1,\qquad [x^n]f_e\ge2\quad(n\ge2).
\end{equation}
The series $a$ and $T$ are even, $V_1=0$, and $Z_0=1$, so
$[x]f_e=1$. Since $[x^0]a=[x^0]T=[x^2]T=1$, the term $TaV$
contributes at least $V_2+V_0\ge2$ at degree two. At degree three,
$TaV$ and $xTZ$ contribute at least $V_3$ and $Z_0$, respectively.
For $n\ge4$, $TaV$ contributes at least $V_n+V_{n-2}\ge2$.

For the odd slices set $\alpha=\phi(x^4)$ and
$\beta=2\psi(x^8)$. Separating the even and odd theta indices gives
\[
 P=\alpha+x\beta,\quad M=\alpha-x\beta,
 \quad a^2=\alpha^2+x^2\beta^2,\quad b^2=2\alpha\beta.
\]
Put $C=2x^3(1-x)Pab/E\ge0$. Equation
\eqref{eq:compensated-fpm} becomes
\begin{align}
 f_+&=\alpha bV+x(\beta aG-\alpha b)+C,
          \label{eq:compensated-fplus-positive}\\
 f_-+2x^2&=x\beta bV+(\alpha aG-x^2\beta b+2x^2)+C.
          \label{eq:compensated-fminus-positive}
\end{align}
To prove that both parenthesized series are nonnegative, note that $\alpha b=2\phi(x^4)\psi(x^4)$
has support at $4m$ and coefficient at most
$2(2\lfloor\sqrt m\rfloor+1)$ there. Each integer square in a
representation determines at most one triangular number.
Since $\beta a\ge2$, for $m\ge2$ we have
\[
 G_{4m}\ge (Q_3)_{4m}+(Q_3)_{4m-1}
 \ge4m-3\ge2\lfloor\sqrt m\rfloor+1.
\]
Here $\lfloor\sqrt m\rfloor\le m-1$ for $m\ge2$.
For $m=0,1$, the required bounds are the equalities
$2G_0=2$ and $2G_4=6$. This proves $\beta aG-\alpha b\ge0$.

Similarly, $\beta b=4\psi(x^8)\psi(x^4)$ has coefficient at
$4m$ at most $4(\lfloor\sqrt m\rfloor+1)$, by counting
solutions of $r(r+1)+s(s+1)/2=m$ with $r,s\ge0$.
For $m\ge3$, \eqref{eq:compensated-partition-lower-bound} gives
\[
 G_{4m+2}\ge4m-1\ge4(\lfloor\sqrt m\rfloor+1).
\]
The last inequality uses $\lfloor\sqrt m\rfloor\le m-2$ for $m\ge3$.
The remaining cases are $m=0,1,2$. At $m=0$,
$[x^2]\alpha aG=2$, and the added $2x^2$ covers
$(\beta b)_0=4$. At $m=1$, use $G_6=5\ge(\beta b)_4=4$;
at $m=2$, use $G_{10}\ge7\ge(\beta b)_8=4$.
Thus $\alpha aG-x^2\beta b+2x^2\ge0$.
Equations \eqref{eq:compensated-fplus-positive} and
\eqref{eq:compensated-fminus-positive} prove
$f_+\ge0$ and $f_-+2x^2\ge0$. The coefficient at $x^2$ in the
latter series is zero, since $V_1=0$ and $C$ starts at degree
three. Consequently the complete negative part of $\mathcal R$ is
\begin{equation}\label{eq:compensated-negative-part}
 -2\sum_{h\in\Z}u^{4h-1}q^{((4h-1)^2+9)/2}.
\end{equation}

It remains to compensate \eqref{eq:compensated-negative-part}
after multiplication by $X$. The series $\mathcal C_2$ is a scalar
linear combination of $U_{5,0}$ and $U_{5,2}$.
Lemma~\ref{lem:common-theta-index-transformations} therefore reduces
positivity to the total BG values $j=-1,0,1$.
Fix such a $j$. Write
$f(r)=r(2r-1)$, $s=4h-1$, and $r=j-s$.
A negative source has multiplicity two and degree
$f(r)+(s^2+9)/2$. For $h\le0$, assign it to the even inner BG
value $t=s+1=4h$ and outer theta index $r-1$; the required extra
$q$-degree in its even slice is
\[
 e=6+4j-20h.
\]
For $h\ge1$, use $t=s-1=4h-2$ and outer index $r+1$;
now $e=20h-2-4j$. Both equations follow by subtracting the
displayed quadratic degrees. Except when $h=0,j=-1$, one has
$e/2\ge2$, so \eqref{eq:compensated-capacity} supplies two
distinct positive labels. In the exception, the source degree
is five. Use one label at $(t,r')=(0,-1)$ and one at
$(t,r')=(-2,1)$, both with $e=2$.
The two target coefficients equal $[x]f_e=1$ and have the correct total degree.

This assignment is injective on labels. The first
ordinary family uses nonpositive multiples of four for $t$;
the second uses positive values congruent to two modulo four.
Within each family, $t$ determines $h$. The exceptional additional
value $t=-2$ belongs to neither family. Thus no positive label is
used twice, and every negative source is covered. Quadratic
starting degrees ensure local finiteness, so
Lemma~\ref{lem:coefficient-allocation} proves
$X\mathcal R=\mathcal C_2\ge0$.

Finally,
\[
 [u^0q^6]\frac{\Theta^{(2)}-q\Theta^{(0)}}{D}=-2.
\]
For an explicit check, the numerator's coefficients at
$u^0q^0,u^0q^2,u^0q^4,u^0q^6$ are $1,1,11,-18$.
The last value is $12-30$: the first contribution is the
$\Theta^{(2)}$ coefficient already computed above, and the second
chooses one $u^{-1}q^3$ term and two $uq$ terms among the five
factors of $X$. The coefficients of $D^{-1}$ at these four
$q$-degrees are $1,1,2,3$, giving $-18+11+2+3=-2$.
The added term $c q^6(1-q^2)\Theta^{(2)}/D$ contributes $c$.
Hence $[u^0q^6]\mathcal C_c=c-2$, excluding $c<2$.
\end{proof}

\section{The first local compensation inequality}\label{sec:local-L1}

Write $X=S(u,q)$, $Y=S(u^{-1},q)$ and $D=(q^2;q^2)_\infty$.

\begin{proposition}\label{prop:local-L1}
One has
\begin{equation}\label{eq:l1-target}
 \frac{X^3Y^2-quXY^4+q^3(u^{-1}X^5-qX^3Y^2)}{D}
 \in\Nzero[u,u^{-1}][[q]]
\end{equation}
\end{proposition}

Put
\begin{equation}\label{eq:l1-four-kernel}
 R_1(u,q)=\frac{(1-q^4)X^2Y^2+q^3u^{-1}X^4-quY^4}{D}.
\end{equation}
The series in \eqref{eq:l1-target} is $XR_1$.
We determine the negative coefficients of $R_1$ and cancel them
against positive coefficients after multiplication by $X$.
The assignment preserves both total degree and BG-rank.

\subsection{The three scalar slices}
Put $x=q^2$, $E=(x;x)_\infty$, $P=\phi(x)$,
$M=\phi(-x)$, and $T_0=\psi(x^2)$, with the notation of
Lemma~\ref{lem:quartic-slices}. We use the common table
\eqref{eq:l1-hadamard-table} and the index transformations of
Lemma~\ref{lem:common-theta-index-transformations}.
Hence there are scalar series $R_0,g_+,g_-$ such that
\begin{equation}\label{eq:l1-slice-reduction}
 [u^s]R_1=
 \begin{cases}
 q^{s^2/2}R_0(x),&s\text{ even},\\
 q^{(s^2+1)/2}g_+(x),&s\equiv1\pmod4,\\
 q^{(s^2+1)/2}g_-(x),&s\equiv-1\pmod4.
 \end{cases}
\end{equation}
Introduce the nonnegative theta series
\[
 a=\phi(x^2),\quad b=2\psi(x^4),\quad
 A_0=\phi(x^4),\quad B_0=2\psi(x^8)
\]
and their differences
\[
 K=a-xb,\quad L=b-a,\quad
 K_2=A_0-x^2B_0,\quad L_2=B_0-A_0.
\]
In particular $P=A_0+xB_0$, $a^2=A_0^2+x^2B_0^2$,
$b^2=2A_0B_0$, and $T_0^2=ab/2$.
Substitution in \eqref{eq:l1-hadamard-table} gives
\begin{align}
 R_0&=\frac{T_0\{(1+x)P^2-4xT_0^2\}}
                   {\prod_{k\ge2}(1-x^k)},\label{eq:l1-even-scalar}\\
 g_+&=\frac{aK_2L+xbL_2K+x(1-x)Pab}{E},\label{eq:l1-oddplus-scalar}\\
 g_-&=\frac{xaL_2L+bK_2K+x(1-x)Pab}{E}.\label{eq:l1-oddminus-scalar}
\end{align}
For example the numerators of $g_\pm$ before dissection are
$\{4(1-x^2)PT_0^2+(x-1)P^3\mp(x+1)M^3\}/2$.
Thus the last two identities also follow by direct substitution of the
displayed duplications and $M=A_0-xB_0$.

\subsection{The even slice and its unique negative coefficient}

We need an elementary observation about multiplication by an Euler
product. Its hypothesis combines all exceptional degrees, so each
positive coefficient is required to cover their total contribution.

\begin{lemma}[Absorbing finitely many defects into an Euler tail]
\label{lem:euler-tail-allocation}
Let $b\ge1$, let $(c_k)_{k\ge b}$ be nonnegative integers, and put
\[
 C(x)=\sum_{k\ge b}c_kx^k,\qquad
 W(x)=\prod_{k\ge b}(1-x^k)^{-c_k}.
\]
Let $P\in\Nzero[x]$ and $A_+\in\Nzero[[x]]$ satisfy
$A_+\ge PC$ coefficientwise. Then
\begin{equation}\label{eq:euler-tail-positive-decomposition}
 (A_+-P)W+P
  =(A_+-PC)W+P\{CW-W+1\}\ge0.
\end{equation}
If $[x^d]P=0$ for $d\ge b$ and $[x^d]A_+=0$ whenever
$[x^d]P>0$, the negative coefficients of $(A_+-P)W$ are exactly
$-[x^d]P$ at those degrees $d$ for which $[x^d]P>0$.
This includes $P=0$.
\end{lemma}
\begin{proof}
The coefficient of $x^n$ in $W$ counts partitions of $n$ with $c_k$
colors available at size $k$. A \emph{colored part type} specifies
both the size and the color, not its multiplicity. The coefficient
of $x^n$ in $CW$ counts such a partition together with a marked
colored part type present in it: remove one occurrence of that type
to obtain the $W$ term, and use the removed type as the $C$ term.
Thus a nonempty colored partition with $r$ distinct colored part
types contributes $r-1\ge0$ to $CW-W+1$; the empty partition
contributes zero. In particular $CW-W+1\in\Nzero[[x]]$.

Expanding the right side of
\eqref{eq:euler-tail-positive-decomposition} proves the identity.
Both terms there are nonnegative by the capacity assumption.
All products are defined degree by degree, since $C$ has positive
order and there are finitely many part sizes and colors at any
fixed degree. Finally $W=1+O(x^b)$. Hence below degree $b$ the
coefficients of $(A_+-P)W$ equal those of $A_+-P$. This proves the
specified negative coefficients, and the first assertion excludes
negative coefficients outside the support of $P$.
\end{proof}

\begin{remark}\label{rem:euler-tail-single-defect}
For a single defect $P=mx^d$, the capacity condition is exactly
$[x^{d+k}]A_+\ge mc_k$ for every $k\ge b$. For several defects it
is the stronger, aggregated condition
\[
 [x^n]A_+\ge\sum_{d\ge0}[x^d]P\,c_{n-d},
 \qquad c_r=0\quad(r<b).
\]
Separate lower bounds for each summand would not suffice: they
could use the same coefficient capacity more than once. In the
application below there is only one defect, $P=x^3$.
\end{remark}

\begin{lemma}\label{lem:l1-even-spectrum}
The only negative coefficient of $R_0$ is $[x^3]R_0=-1$.
\end{lemma}
\begin{proof}
Using \eqref{eq:l1-classical-identities}, write
\begin{equation}\label{eq:l1-even-product}
 R_0=\mathcal V(a^2+xL^2+x^2b^2),\qquad
 \mathcal V=\prod_{\substack{k\ge3\\ k\text{ odd}}}(1-x^k)^{-1}
 \prod_{\substack{k\ge2\\ k\equiv2\ (4)}}(1-x^k)^{-2}.
\end{equation}
The identity in the numerator is
$(1+x)P^2-4xT_0^2=a^2+x(b-a)^2+x^2b^2$.
The elementary theta telescope is
\[
 L=\sum_{m\ge1}t_m,\qquad
 t_m=x^{2m(m-1)}(1-x^{2m})^2.
\]
We claim $\mathcal Vt_mt_n\ge0$ if at least one of $m,n$ lies
outside $\{1,2,4\}$. By symmetry take this index to be $m$.
Absorb the $n$ square using the two size-$2$ Euler factors. If $m=3$,
use the two size-$6$ factors for its square. If $d=2m\ge12$, use the
disjoint odd pairs $\{3,d-3\}$ and $\{5,d-5\}$ and the identity
\begin{equation}\label{eq:l1-pair-absorption}
 \frac{1-x^{r+s}}{(1-x^r)(1-x^s)}
 =\frac1{1-x^r}+\frac{x^s}{1-x^s}\ge0.
\end{equation}
This is also the $L=1$ specialization of
\cite[Proposition~4.3]{BG08eta}, with $z_1=x^r$ and $z_2=x^s$.
All four sizes are distinct and at least $3$.
For $d=10$ use $\{3,7\}$ and the single size $5$, since
$(1-x^{10})/(1-x^5)=1+x^5$. These factorizations use no size-$2$
factor beyond those already assigned to $n$, and prove the claim.
Here and in the odd-slice argument below, distinct Euler-factor
copies are required within each summand; a repeated size uses its
two available copies. Different summands are separate coefficientwise
inequalities and may each use the same denominator.

Put $L_*=t_1+t_2+t_4$ and
\begin{equation}\label{eq:l1-even-finite-numerator}
 P_*=1+8x^2+xL_*^2,\qquad
 Q=\frac{P_*}{(1-x^2)^2(1-x^3)}.
\end{equation}
The claim, together with $a^2+x^2b^2\ge1+8x^2$, shows that
$R_0-\mathcal VP_*$ is nonnegative and vanishes through degree $3$.
The rational identity in Appendix~\ref{app:l1-certificates} proves
\begin{equation}\label{eq:l1-even-rational-bounds}
 [x^3]Q=-1,\quad [x^n]Q\ge0\ (n\ne3),\quad
 [x^n]Q\ge2\ (n\ge4).
\end{equation}
To include the remaining Euler factors, write
\[
 \mathcal VP_*=QW,\qquad
 W=\prod_{\substack{k\ge5\\ k\text{ odd}}}(1-x^k)^{-1}
   \prod_{\substack{k\ge6\\ k\equiv2\ (4)}}(1-x^k)^{-2}.
\]
Apply Lemma~\ref{lem:euler-tail-allocation} with $b=5$, $P=x^3$,
and $A_+=Q+x^3$. Here $c_k=1$ for odd $k\ge5$,
$c_k=2$ for $k\ge6$ congruent to two modulo four, and $c_k=0$
otherwise. Equation~\eqref{eq:l1-even-rational-bounds} supplies
$[x^{k+3}]A_+\ge2\ge c_k$. Thus $QW+x^3\ge0$, and the
coefficient of $x^3$ in $QW$ is exactly $-1$.
Adding $R_0-\mathcal VP_*$, which is
nonnegative and starts in degree $4$, proves the lemma.
\end{proof}

\subsection{The odd slices and positive capacities}
\begin{lemma}\label{lem:l1-odd-capacity}
For every $n\ge0$,
\[
 [x^n]g_+\ge1,\qquad [x^n]g_-\ge1,\qquad [x^4]g_-\ge3.
\]
\end{lemma}
\begin{proof}
We use the four telescopes
\begin{align*}
 L&=\sum_{m\ge1}x^{2m(m-1)}(1-x^{2m})^2,&
 K&=\sum_{n\ge0}x^{2n^2}(1-x^{2n+1})^2,\\
 L_2&=\sum_{m\ge1}x^{4m(m-1)}(1-x^{4m})^2,&
 K_2&=\sum_{n\ge0}x^{4n^2}(1-x^{4n+2})^2.
\end{align*}
Each follows by expanding a square and shifting one index; the first
and third are the same identity at $x$ and $x^2$, as are the second
and fourth. Put $K_0=(1-x)^2$ and $K_{\rm tail}=K-K_0$.

We first record which pairs of squares can be absorbed by distinct
Euler factors in $E$. For even $d,e\ge2$ with $\max(d,e)\ge6$,
\begin{equation}\label{eq:l1-two-even-squares}
 (1-x^d)^2(1-x^e)^2/E\ge0.
\end{equation}
Take $d=\max(d,e)$. Use factors at sizes $1,e$ for the $e$ square.
For $d\ge12$ use the odd pairs $\{3,d-3\}$ and $\{5,d-5\}$
for the remaining square; for $d=10$ use $\{3,7\}$ and $5$.
The finite endpoints use the following two unused divisor sizes:
\[
\begin{array}{c|c|c}
 d&e&\text{divisor sizes}\\\hline
 6&2&3,6\\6&4\text{ or }6&2,3\\
 8&2&4,8\\8&4&2,8\\8&6\text{ or }8&2,4
\end{array}
\]
Every individual quotient by a divisor is a finite geometric polynomial.
All assigned sizes are distinct within a summand, since $E$ has one
factor at each size; \eqref{eq:l1-pair-absorption} handles
an assigned pair whose sum is $d$.
For odd $d\ge3$ and even $e\ge4$, use $1,d$ for the odd square and
$2,e$ for the even square. For $e=2$ and odd $d\ge7$, use $1,2$
for the even square, $d$ for one remaining factor, and $\{3,d-3\}$
for the other. This proves the corresponding nonnegativity in both
mixed cases. The two exceptions $e=2$, $d=3,5$ are retained below.

Define the finite polynomials
\begin{align}
 P_0&=(1-x^2)^2\{(1-x^2)^2+x^4(1-x^4)^2\},\nonumber\\
 P_2&=(1-x^4)^2\{(1-x^2)^2+x^4(1-x^4)^2\},\label{eq:l1-odd-bad-blocks}\\
 P_3&=(1-x^2)^2\{x^2(1-x^3)^2+x^8(1-x^5)^2\}.\nonumber
\end{align}
The square absorptions imply
\begin{equation}\label{eq:l1-odd-truncation}
 \frac{K_2L}{E}\ge\frac{P_0}{E},\quad
 \frac{L_2L}{E}\ge\frac{P_2}{E},\quad
 \frac{L_2K_{\rm tail}}E\ge0,\quad
 \frac{K_2K_{\rm tail}}E\ge\frac{P_3}E.
\end{equation}
Here the only exceptions in $K_2L$ have square sizes $(2,2)$ or
$(2,4)$, giving $P_0$. Those in $L_2L$ have sizes $(4,2)$ or
$(4,4)$, giving $P_2$. Every term of $L_2K_{\rm tail}$ has an even
size at least $4$ and an odd size at least $3$. In
$K_2K_{\rm tail}$ the only exceptions have sizes $(2,3)$ or
$(2,5)$, giving $P_3$. This accounts for every double-telescope term.
The quadratic initial degrees make all sums locally finite.

Introduce
\begin{align}
 H_1&=\{P_0+(2x+4x^2)(1-x)\}/E,&
 H_0&=\{P_0+(4x+2x^2)(1-x)\}/E,\nonumber\\
 I&=\{P_2+(2+4x)(1-x)\}/E,&
 J_0&=\{P_3+(1+2x^3)(1-x)\}/E,\label{eq:l1-five-kernels}\\
 H&=\{xP_2+2P_3+(2+4x^3)(1-x)\}/E.\nonumber
\end{align}
The five rational identities in
Appendix~\ref{app:l1-certificates} show that all these kernels are
nonnegative, with
\begin{equation}\label{eq:l1-five-capacities}
 (H_1)_0=1,\quad (H_0)_2=0,\quad (H_0)_n\ge1\ (n\ne2),
 \quad H_n\ge1\ (n\ge0),\quad H_4\ge3.
\end{equation}
Subscripts here denote coefficients in $x$.

For $g_+$, combine the $K_0$ summand with the last term of
\eqref{eq:l1-oddplus-scalar}. The identity
\[
 (1-x)L_2+Pa=P(a-1)+B_0+xA_0
\]
and \eqref{eq:l1-odd-truncation} give
\[
 g_+\ge\frac{aP_0+xb(1-x)\{P(a-1)+B_0+xA_0\}}{E}
       \ge(a-1)H_1+H_0.
\]
To verify the last inequality, the difference is $(1-x)/E$ times
\[
 x(a-1)\{bP-(2+4x)\}+x(bB_0-4)+x^2(bA_0-2),
\]
a nonnegative series. We used $a\ge1$, $b,B_0\ge2$, $A_0\ge1$,
$bP\ge2+4x$, and $(1-x)/E\ge0$.
By \eqref{eq:l1-five-capacities}, $H_0$ has coefficient at least one at
every degree except $2$. At degree $2$, $(a-1)H_1$ supplies at least
$(a-1)_2(H_1)_0=2$. This proves the bound for $g_+$.

For $g_-$, similarly use
\[
 (1-x)K_2+xPa=A_0+xP(a-1)+x^3B_0
\]
to obtain
\[
 g_-\ge\frac{xaP_2+bP_3+b(1-x)\{A_0+xP(a-1)+x^3B_0\}}{E}
       \ge x(a-1)I+(b-2)J_0+H.
\]
The difference in the last inequality is $(1-x)/E$ times
\[
 b(A_0-1)+x^3b(B_0-2)+x(a-1)\{bP-(2+4x)\}\ge0.
\]
The bounds for $H$ in \eqref{eq:l1-five-capacities} finish the proof.
\end{proof}

\subsection{Restoring the fifth theta factor}
\begin{proof}[Proof of Proposition~\ref{prop:local-L1}]
By Lemmas~\ref{lem:l1-even-spectrum} and \ref{lem:l1-odd-capacity},
the negative terms of $R_1$ are precisely
\begin{equation}\label{eq:l1-negative-atoms}
 -u^{2h}q^{2h^2+6}\qquad(h\in\Z).
\end{equation}
Fix a total BG exponent $j\in\{-1,0,1\}$ and put $f(r)=r(2r-1)$.
After multiplication by $X$, the negative term indexed by $h$ has
outer theta index $r=j-2h$ and total degree
\[
 N=f(j-2h)+2h^2+6.
\]
We assign it to a positive label from an odd slice $s=t$ of $R_1$,
with outer theta index $r'=j-t$. Its degree is
$f(j-t)+(t^2+1)/2+e$, where the available coefficient is
$[x^{e/2}]g_+$ or $[x^{e/2}]g_-$ according to $t$ modulo $4$.
Choose
\begin{equation}\label{eq:l1-outer-injection}
\begin{array}{c|c|c}
 \text{source}&t&e\\\hline
 h\le0&2h+1&2+4j-10h\\
 h\ge1&2h-1&4-4j+10h
\end{array}
\end{equation}
except that for $(j,h)=(-1,0)$ we choose $t=-1$, $e=8$.
Direct subtraction gives
\[
 f(j-2h)+2h^2+6=f(j-t)+(t^2+1)/2+e
\]
in both rows and in the exceptional case.
For $h\le-1$ one has $e\ge8$, and for $h\ge1$ one has $e\ge10$.
For $h=0$, $j=0,1$ gives $e=2,6$ respectively. Hence all assigned
coefficients exist and are at least $1$ by
Lemma~\ref{lem:l1-odd-capacity}.

It remains to check that the target coefficients have sufficient
multiplicity. Within each row of \eqref{eq:l1-outer-injection}, the target index
$t$ determines the source. The only common $t$ of the two rows is $t=1$,
coming from $h=0$ and $h=1$. Their degrees $e=2+4j$ and $14-4j$
would agree only for $j=3/2$, which is excluded. The exceptional
assignment $(j,h)=(-1,0)$ coincides only with the assignment from
$(j,h)=(-1,-1)$: both have $t=-1,e=8$. This target therefore
receives exactly two sources, while $[x^4]g_-\ge3$ supplies at least
three positive labels. Assign the two sources distinct labels.
Thus the assignment is injective and preserves both $u$-
and $q$-degree. Every unassigned contribution is nonnegative, proving $[u^j]XR_1\ge0$ for $j=-1,0,1$.

The labeled assignment is locally finite and hence satisfies
Lemma~\ref{lem:coefficient-allocation}. Finally $XR_1$ is a scalar
linear combination of $U_{5,0},U_{5,2},U_{5,4}$, so
Lemma~\ref{lem:common-theta-index-transformations} extends the
three representative inequalities to every integer BG value.
This proves \eqref{eq:l1-target}.
\end{proof}

\section{The second local compensation inequality}\label{sec:local-L2}
\label{subsec:second-compensation}

The next inequality combines $\Theta^{(4)}-q\Theta^{(2)}$ with
the exceptional difference $\Theta^{(2)}-q\Theta^{(0)}$.

\begin{proposition}\label{prop:local-L2}
The series
\[
 \mathcal L_2(u,q)=
 \frac{\Theta^{(4)}+(q^3-q)\Theta^{(2)}-q^4\Theta^{(0)}}{D(q)}
\]
has nonnegative coefficients.
\end{proposition}

Write $X=S(u,q)$, $Y=S(u^{-1},q)$, and $x=q^2$. Put
\[
 U_0=u^{-1}X^4,\qquad U_2=X^2Y^2,\qquad U_4=uY^4,
 \qquad
 \mathcal R_2=\frac{U_4+(q^3-q)U_2-q^4U_0}{D}.
\]
Thus $\mathcal L_2=X\mathcal R_2$. We first determine the complete
negative part of $\mathcal R_2$ and establish a lower bound for its
even BG slices. We then compensate the negative coefficients after
multiplication by $X$.

Retain the notation of Lemma~\ref{lem:common-scalar-bounds}, and put
\[
 P=\phi(x),\quad M=\phi(-x),\quad T=\psi(x^2),\quad
 \alpha=\phi(x^4),\quad\beta=2\psi(x^8).
\]
The elementary duplication identities are
\begin{equation}\label{eq:l2-duplication}
 \begin{gathered}
 P=\alpha+x\beta,\qquad M=\alpha-x\beta,\qquad
 a^2=\alpha^2+x^2\beta^2,\qquad b^2=2\alpha\beta,\\
 P^2=a^2+xb^2,\qquad T^2=ab/2.
 \end{gathered}
\end{equation}
They follow by separating the parity classes of the theta indices.
\begin{lemma}\label{lem:l2-quartic-spectrum}
There are scalar series $g,F_+,F_-$ such that
\begin{equation}\label{eq:l2-slices}
 [u^s]\mathcal R_2=
 \begin{cases}
 q^{s^2/2+1}g(x),&s\text{ even},\\
 q^{(s^2-1)/2}F_+(x),&s\equiv1\pmod4,\\
 q^{(s^2-1)/2}F_-(x),&s\equiv-1\pmod4.
 \end{cases}
\end{equation}
They satisfy
\begin{equation}\label{eq:l2-scalar-capacity}
 g\ge0,\qquad g_0=3,\quad g_1=0,\quad g_n\ge2\ (n\ge2),
\end{equation}
\begin{equation}\label{eq:l2-odd-spectrum}
 F_++x\ge0,\qquad [x]F_+=-1,\qquad
 F_-+2x^4\ge0,\qquad [x^4]F_-=-2.
\end{equation}
Consequently the complete negative part of $\mathcal R_2$ is
\begin{equation}\label{eq:l2-negative-part}
 -\sum_{h\in\mathbb Z}u^{4h+1}q^{((4h+1)^2+3)/2}
 -2\sum_{h\in\mathbb Z}u^{4h-1}q^{((4h-1)^2+15)/2}.
\end{equation}
\end{lemma}

\begin{proof}
Substitute the common quartic table
\eqref{eq:l1-hadamard-table} into $\mathcal R_2$.
Lemma~\ref{lem:common-theta-index-transformations} extends the
three representative coefficients to every integer $s$ and gives
\eqref{eq:l2-slices}, where
\begin{align}
 g&=\frac{(1-x)T\{4(1+x)T^2-P^2\}}{E},
       \label{eq:l2-even-formula}\\
 F_+&=\frac{(P^3+M^3)/2-x(1-x)Pab-x^2(P^3-M^3)/2}{E},
       \label{eq:l2-plus-formula}\\
 F_-&=\frac{(P^3-M^3)/2-x(1-x)Pab-x^2(P^3+M^3)/2}{E}.
       \label{eq:l2-minus-formula}
\end{align}
Using \eqref{eq:l2-duplication}, the two odd formulas factor as
\begin{align}
 F_+&=\alpha(a+x^2b)K/E+x^2\beta(b+xa)G,
       \label{eq:l2-plus-factor}\\
 F_-&=x\alpha(b+xa)G+x\beta(a+x^2b)K/E.
       \label{eq:l2-minus-factor}
\end{align}
For example, the two summands in the numerator of
\eqref{eq:l2-plus-factor} are
$\alpha(a-xb)(a+x^2b)$ and $x^2\beta(b-a)(b+xa)$.
Expanding them and using $a^2=\alpha^2+x^2\beta^2$ and
$b^2=2\alpha\beta$ proves \eqref{eq:l2-plus-formula}.
The factors are exchanged, with one factor $x$, in the minus case.

We first prove the two assertions in \eqref{eq:l2-odd-spectrum}.
Let $H_+=\alpha(a+x^2b)$. Then
\begin{equation}\label{eq:l2-plus-subtraction}
 F_+=H_+V+x^2\beta(b+xa)G-xH_+.
\end{equation}
The series $H_+$ is even and satisfies
\[
 (H_+)_{2k}\le8\lfloor\sqrt{k/2}\rfloor+4\qquad(k\ge0).
\]
Indeed, $\alpha a$ counts $2r^2+s^2=k$, with $r,s\in\mathbb Z$,
so its coefficient is at most
$2(2\lfloor\sqrt{k/2}\rfloor+1)$.
For $k\ge1$, the coefficient contributed by $x^2\alpha b$ is at
most $2(2\lfloor\sqrt{(k-1)/2}\rfloor+1)$: the square index
determines at most one nonnegative triangular index. At $k=0$
this second contribution vanishes.
For $k\ge4$, the middle term of \eqref{eq:l2-plus-subtraction}
contributes at degree $2k+1$ at least
\[
 4G_{2k-1}\ge4\{[x^{2k-1}]Q_3+[x^{2k-2}]Q_3\}
 \ge8k-16\ge8\lfloor\sqrt{k/2}\rfloor+4.
\]
At $k=4$ the last inequality is $16\ge12$. For $k\ge5$,
$\sqrt{k/2}\le k-3$ gives
$8\lfloor\sqrt{k/2}\rfloor+4\le8k-20\le8k-16$.
The small subtracted coefficients at $k=0,1,2,3$ are respectively
$1,4,2,10$.
At $k=1,2,3$ they are covered by the following lower bounds for
the middle term:
\[
 2(2G_1+G_0)=6,\qquad 4G_3=4,\qquad4G_5=12.
\]
At $k=0$ both positive terms have coefficient zero at $x$, since
$V_1=0$. At even degrees there is no subtraction. This proves
$F_++x\ge0$ and $[x]F_+=-1$.

Next put $H_-=\beta(a+x^2b)$ and $P_-=\alpha(b+xa)G$. Then
\begin{equation}\label{eq:l2-minus-subtraction}
 F_-=x\{P_-+H_-V-xH_-\}.
\end{equation}
The even series $H_-$ satisfies
\[
 (H_-)_{2k}\le8\lfloor\sqrt{k/2}\rfloor+8\qquad(k\ge0).
\]
For $\beta a$, the equation is $2r(r+1)+s^2=k$, where
$r\ge0$ and $s\in\mathbb Z$. There are at most
$\lfloor\sqrt{k/2}\rfloor+1$ possible $r$, at most two possible
$s$, and a prefactor two. For $x^2\beta b$, the analogous
bound with $(k-1)/2$ counts two nonnegative triangular indices,
with at most one second index for each first index, and prefactor
four. At $k=0$ this second contribution vanishes.
At degree $2k+1$ in the braces in
\eqref{eq:l2-minus-subtraction}, the first term supplies at least
$2G_{2k+1}+G_{2k}$. For $k\ge4$,
\[
 2G_{2k+1}+G_{2k}\ge6k-7
 \ge8\lfloor\sqrt{k/2}\rfloor+8.
\]
At $k=4$ the last inequality is $17\ge16$. For $k\ge5$,
$\sqrt{k/2}\le k/3$ follows from $k\ge9/2$, and
$6k-7-(8k/3+8)=(10k-45)/3\ge0$.
The remaining values $k=0,1,2,3$ have
$(H_-)_{2k}=2,8,0,4$, respectively. At $k=0$,
$(P_-)_1=2G_1+G_0=3$. At $k=3$, use
$(P_-)_7\ge2G_7+G_6=19$. At $k=1$ the exact coefficient
in the braces is
\[
 (P_-)_3+(H_-V)_3-(H_-)_2
 =(2G_3+G_2+2G_0)+(2V_3+8V_1)-8=-2.
\]
There is no subtraction at the other parity. This proves
$F_-+2x^4\ge0$ and $[x^4]F_-=-2$.

It remains to prove \eqref{eq:l2-scalar-capacity}. Set
\[
 Q_2(x)=\prod_{n\ge2}(1-x^n)^{-1},\qquad
 H=4(1+x)T^2-P^2,
\]
so that $g=TQ_2H$. With $d(r,s)=r(r+1)+s(s+1)$, splitting
the two square indices into both nonzero, exactly one nonzero,
and both zero gives the exact identity
\begin{equation}\label{eq:l2-even-pairing}
 H=4\sum_{r,s\ge0}x^{d(r,s)}(1-x^{r+s+2})
   +4x\sum_{r,s\ge0}x^{d(r,s)}
   -4\sum_{r\ge1}x^{r^2}-1.
\end{equation}
The first sum pairs the theta tuple $(r,s)$ with the four sign
choices of the square tuple $(r+1,s+1)$. In the second sum,
separating $s=0$ and combining that part with the axis terms gives
\[
 4\sum_{r\ge1}x^{r(r-1)+1}(1-x^{r-1}).
\]
The $r=1$ term vanishes, the $r=2$ term is $4x^3(1-x)$,
and each term with $r\ge3$ becomes nonnegative after
multiplication by $Q_2$, by canceling its Euler factor of size
$r-1$. Likewise each term of the first sum in
\eqref{eq:l2-even-pairing} becomes nonnegative after multiplication
by $Q_2$, since $r+s+2\ge2$.

Extract the pairs $(r,s)=(0,0),(1,0)$ from the first sum,
the exceptional $r=2$ axis term, the pair $(r,s)=(0,1)$ from
the remaining $s\ge1$ part of the second sum, and the origin
subtraction. Their sum is
\[
 4(1-x^2)+4x^2(1-x^3)+4x^3(1-x)+4x^3-1
 =3+8x^3-4x^4-4x^5=:p(x).
\]
All remaining terms are nonnegative after multiplication by
$Q_2$. This formal decomposition is locally finite because the
theta starting degrees are quadratic.

The rational remainder has the coefficient formula
\begin{equation}\label{eq:l2-rational-certificate}
 F(x):=\frac{p(x)}{(1-x^2)(1-x^3)(1-x^4)},\qquad
 [x^{12m+r}]F=a_r+b_rm+18\binom m2
 \quad(m\ge0,\ 0\le r<12),
\end{equation}
where
\[
\begin{array}{c|rrrrrrrrrrrr}
 r&0&1&2&3&4&5&6&7&8&9&10&11\\\hline
 a_r&3&0&3&11&2&7&13&14&8&21&15&24\\
 b_r&18&27&21&30&24&33&27&36&30&39&33&42.
\end{array}
\]
Multiplying $F$ by $(1-x^{12})^3$ gives
\[
 p(x)(1+x^2+x^4+x^6+x^8+x^{10})
      (1+x^3+x^6+x^9)(1+x^4+x^8).
\]
For each residue $r$, let $A_r,B_r,C_r$ be the coefficients of
this degree-$32$ polynomial at $r,r+12,r+24$, respectively,
with missing coefficients zero. Direct multiplication gives
\[
\begin{array}{c|rrrrrrrrrrrr}
 r&0&1&2&3&4&5&6&7&8&9&10&11\\\hline
 A_r&3&0&3&11&2&7&13&14&8&21&15&24\\
 B_r&12&27&15&8&20&19&1&8&14&-3&3&-6\\
 C_r&3&-9&0&-1&-4&-8&4&-4&-4&0&0&0.
\end{array}
\]
The residue-tail formula of Lemma~\ref{lem:residue-tail} applies
with period $12$, three denominator factors, and a cleared numerator
of degree $32$. The cutoff is zero, so the formula holds for every $m\ge0$.
Explicitly, the coefficient of $y^m$ in
$(A_r+B_ry+C_ry^2)/(1-y)^3$ is
$A_r+(2A_r+B_r)m+(A_r+B_r+C_r)\binom m2$.
Here $A_r=a_r$, $2A_r+B_r=b_r$, and $A_r+B_r+C_r=18$.
This proves \eqref{eq:l2-rational-certificate} at every degree.
In particular $F_0=3$, $F_1=0$, and $F_n\ge2$ for $n\ge2$.
As
\[
 pQ_2=F\prod_{n\ge5}(1-x^n)^{-1},
\]
the same lower bounds hold for $pQ_2$, and hence for $Q_2H$.
Multiplication by $T\ge1$ proves the required inequalities for
$g$. Its exact constant and linear coefficients are $3$ and $0$
by \eqref{eq:l2-even-formula}. This proves
\eqref{eq:l2-scalar-capacity} and completes the lemma.
\end{proof}

\begin{proof}[Proof of Proposition~\ref{prop:local-L2}]
By Lemma~\ref{lem:common-theta-index-transformations}, applied to
the scalar combination $\mathcal L_2$ of the three $U_{5,k}$,
it suffices to prove positivity at total BG values $j=-1,0,1$.
Fix such a value and put $f(r)=r(2r-1)$.
The two negative source types in \eqref{eq:l2-negative-part},
after multiplication by the outer $X$, are
\[
\begin{array}{c|c|c|c}
 \text{type}&\text{inner BG }s&\text{multiplicity}&\text{total degree}\\\hline
 +&1+4h&1&f(j-s)+(s^2+3)/2\\
 -&-1+4h&2&f(j-s)+(s^2+15)/2.
\end{array}
\]
An even inner BG value $t$ and outer index $r'=j-t$ supply
positive labels at degrees
$f(r')+t^2/2+1+e$, where $e=2m$ and the capacity is $g_m$.
Assign the negative sources by the following table:
\begin{equation}\label{eq:l2-outer-injection}
\begin{array}{c|c|c|c}
 \text{type}&\text{range}&t&e\\\hline
 +&h\ge0&s-1=4h&4-4j+20h\\
 +&h\le-1&s+1=4h+2&4j-8-20h\\
 -&h\le0&s+1=4h&4j+8-20h\\
 -&h\ge1&s-1=4h-2&20h-4j.
\end{array}
\end{equation}
Each row preserves total BG and degree, by direct subtraction of
the displayed quadratics. For $j\in\{-1,0,1\}$, every $e/2$
is even and is either zero or at least two. The capacities
$g_0=3$ and $g_m\ge2$ for $m\ge2$ therefore cover both source
multiplicities.

No positive label is used twice. The four target BG sets are,
respectively, nonnegative multiples of four, negative integers
congruent to two modulo four, nonpositive multiples of four, and
positive integers congruent to two modulo four. Their only overlap
is $t=0$, from the two sources with $h=0$. Those sources require
extra degrees $4-4j$ and $8+4j$, which are distinct for every
integer $j$, so they use different coefficients. Within each row, $t$
determines $h$ uniquely. Use two distinct labels of the target coefficient
when the source has multiplicity two.

By Lemma~\ref{lem:l2-quartic-spectrum}, every remaining
coefficient of $\mathcal R_2$ is nonnegative. The assignment
\eqref{eq:l2-outer-injection} covers all negative sources by
unused positive labels of equal total BG and degree. All sums are
locally finite because their starting degrees are quadratic.
Lemma~\ref{lem:coefficient-allocation} therefore proves
$X\mathcal R_2\ge0$ at the three representative BG values,
and Lemma~\ref{lem:common-theta-index-transformations} proves the
assertion for all BG values.
\end{proof}

\section{Symmetries, consequences, and examples}\label{sec:consequences}

\subsection{Reciprocity and BG-rank support}

Reflection gives reciprocity, and the total mass
agrees with~\cite[Theorem 4.1]{BG08}.

\begin{corollary}\label{cor:reciprocity}
For every $n\in\Nzero$ and $j\in\Z$,
\begin{equation}\label{eq:reciprocity}
 K_{n,j}(z)=z^4K_{n,j}(z^{-1}),
 \qquad H_{n,j}(z)=H_{n,j}(z^{-1}),
\end{equation}
and
\begin{equation}\label{eq:total-mass}
 H_{n,j}(1)=\frac{p_j(5n+4)}5.
\end{equation}
\end{corollary}

\begin{proof}
The reflection $\iota(a_0,\ldots,a_4)=(a_4,\ldots,a_0)$
preserves $\A,Q$ and the edge-sign multiset, while
$L(\iota a)=4-L(a)$. Reindexing \eqref{eq:Phi-Psi} therefore gives
$\Psi(u,z,q)=z^4\Psi(u,z^{-1},q)$ and the identity for $K$.
Using $[5]_{z^{-1}}=z^{-4}[5]_z$ gives the identity for $H$;
evaluating $K_{n,j}=[5]_zH_{n,j}$ at $z=1$ gives \eqref{eq:total-mass}.
\end{proof}

The two-core enumeration \eqref{eq:bg-pairs} also determines the support.

\begin{corollary}\label{cor:admissible}
The polynomial $H_{n,j}$ is nonzero exactly when
\begin{equation}\label{eq:BG-support}
 j\equiv n\pmod2,\qquad j(2j-1)\le5n+4.
\end{equation}
With the convention in~\eqref{eq:bg-pairs},
\begin{equation}\label{eq:mass-pairs}
 H_{n,j}(1)=\frac15p^{(2)}\!\left(\frac{5n+4-j(2j-1)}2\right).
\end{equation}
\end{corollary}

\begin{proof}
Combine \eqref{eq:total-mass} with \eqref{eq:bg-pairs}.
For $m\in\Nzero$, $p^{(2)}(m)>0$, while it vanishes otherwise.
Its argument lies in $\Nzero$ exactly under \eqref{eq:BG-support}.
Since $H_{n,j}$ has nonnegative coefficients, it vanishes exactly
when $H_{n,j}(1)=0$.
\end{proof}

\subsection{Affine symmetries from Fayers' action}\label{subsec:affine-symmetries}

We use Fayers' level-$t$ action at $(s,t)=(2,5)$
\cite[Section~3 and Lemma~3.1]{Fayers14};
see also~\cite[Section~5.1]{Fayers23}.
The infinite dihedral group
$W=\langle w_0,w_1\mid w_0^2=w_1^2=1\rangle$
(of type $\widetilde A_1$) acts on charge-zero beta-sets by
$\mathcal B_{w_i\lambda}=w_i(\mathcal B_\lambda)$, where
\begin{equation}\label{eq:fayers-reflections}
 \begin{aligned}
 w_0(x)&=\begin{cases}x-5,&x\text{ even},\\x+5,&x\text{ odd},\end{cases}
 &w_1(x)&=\begin{cases}x+5,&x\text{ even},\\x-5,&x\text{ odd}.\end{cases}
 \end{aligned}
\end{equation}
These are the maps of~\cite[Section~3]{Fayers14} with
$(s-1)(t-1)/2=2$.

\begin{lemma}[Fayers' action, specialized]\label{lem:fayers-specialization}
The maps in~\eqref{eq:fayers-reflections} induce bijections on
partitions, preserve the five-core and integer crank, and
interchange the two-quotient components. For $j=\BG(\lambda)$,
\begin{equation}\label{eq:fayers-bg-reflections}
 \BG(w_0\lambda)=3-j,\qquad
 \BG(w_1\lambda)=-2-j.
\end{equation}
With composition read from right to left, $\tau=w_0w_1$ satisfies
\begin{equation}\label{eq:fayers-position-translation}
 \tau^k(x)=x+10k(-1)^x\qquad(k\in\Z).
\end{equation}
\end{lemma}

\begin{proof}
The two runner level sets
$E=\{a:2a\in\mathcal B_\lambda\}$ and
$O=\{a:2a+1\in\mathcal B_\lambda\}$ have charges $(j,-j)$.
Writing $B+h=\{a+h:a\in B\}$, \eqref{eq:fayers-reflections} gives
\[
 (E,O)\xrightarrow{\ w_0\ }(O+3,E-3),\qquad
 (E,O)\xrightarrow{\ w_1\ }(O-2,E+2).
\]
Since $\operatorname{ch}(B+h)=\operatorname{ch}(B)+h$,
the new charges are $(3-j,j-3)$ and $(-2-j,j+2)$.
They sum to zero, so define partitions. Both maps are involutions,
and runner translation leaves the quotient partition unchanged,
proving the interchange and \eqref{eq:fayers-bg-reflections}.

On five-runner $r$, write $x=r+5a$. The induced maps
$a\mapsto a\mp(-1)^{r+a}$ swap adjacent pairs of levels.
Only finitely many swaps change occupancy; each preserves charge,
including across zero. Thus the five-core and the integer crank
$1+c_1+c_2-2c_4$ are preserved. Finally,
$w_0w_1(x)=x+10(-1)^x$ preserves parity, giving
\eqref{eq:fayers-position-translation} for every integer iterate.
\end{proof}

\begin{corollary}[Translation consequence of Fayers' action]\label{cor:bg-translation}
Write $j=r+5k$, where $r\in\{-2,-1,0,1,2\}$ and $k\in\Z$, and set
\begin{equation}\label{eq:bg-shift}
 \Delta=10k^2+(4r-1)k
        =\frac{j(2j-1)-r(2r-1)}5.
\end{equation}
Then $\Delta\in\Nzero$, and for all $n\in\Nzero$,
\begin{equation}\label{eq:five-families}
 K_{n,j}(z)=K_{n-\Delta,r}(z),\qquad
 H_{n,j}(z)=H_{n-\Delta,r}(z),
\end{equation}
with negative first indices interpreted as zero.
The first identity is induced by $\tau^k$, preserving the five-core
and every integer crank exponent.
\end{corollary}

\begin{proof}
The map $\tau^k$ preserves the two-quotient and shifts BG-rank
by $5k$. Since the two-core of rank $j$ has size $C(j)=j(2j-1)$,
for $\BG(\lambda)=r$ one has
\[
 |\tau^k\lambda|-|\lambda|
   =C(r+5k)-C(r)=5\{10k^2+(4r-1)k\}=5\Delta.
\]
The crank-preserving bijection $\tau^k$, with inverse $\tau^{-k}$,
proves the identity for $K$; division by $[5]_z$ gives that for $H$.
Also $|4r-1|\le9$ implies $\Delta\ge10k^2-9|k|\ge0$.
If $n-\Delta<0$, an inverse image would have negative size,
justifying the zero convention.
\end{proof}

The reflections reduce these five representatives to three.

\begin{corollary}[Reflection consequences and three representatives]\label{cor:bg-reflections}
For every $n\in\Nzero$ and $j\in\Z$,
\begin{equation}\label{eq:bg-reflection-distributions}
 \begin{aligned}
 K_{n,j}(z)&=K_{n+3-2j,\,3-j}(z),\\
 K_{n,j}(z)&=K_{n+2j+2,\,-2-j}(z).
 \end{aligned}
\end{equation}
The same identities hold for $H$, with negative first indices
interpreted as zero. In particular,
\begin{equation}\label{eq:bg-three-representatives}
 K_{n,2}=K_{n-1,1},\qquad K_{n,-2}=K_{n-2,0},
\end{equation}
and similarly for $H$. Together with Corollary~\ref{cor:bg-translation},
these reduce all BG-ranks to $-1,0,1$.
\end{corollary}

\begin{proof}
The involutions preserve two-quotient size and integer crank, while
\[
 C(3-j)-C(j)=5(3-2j),\qquad
 C(-2-j)-C(j)=5(2j+2),\qquad C(j)=j(2j-1).
\]
This proves \eqref{eq:bg-reflection-distributions}; a negative
target size forces the source class to be empty.
Division by $[5]_z$ gives the identities for $H$.
Taking $j=2$ and $j=-2$, respectively, gives
\eqref{eq:bg-three-representatives}.
\end{proof}

\subsection{Cumulative crank inequalities}

For $k\in\Z$, put
\[
 M_{n,j}(k)=\#\{\lambda\vdash5n+4:
                 \BG(\lambda)=j,\ \cc(\lambda)=k\},
 \qquad h_{n,j}(k)=[z^k]H_{n,j}(z).
\]

\begin{corollary}\label{cor:cumulative}
For $n\in\Nzero$ and $j,k\in\Z$,
\begin{equation}\label{eq:cumulative}
 h_{n,j}(k)=\sum_{r\ge0}
       \bigl(M_{n,j}(k-5r)-M_{n,j}(k-5r-1)\bigr)\ge0,
\end{equation}
where the sum is finite.
\end{corollary}

\begin{proof}
The identity $(1-z)K_{n,j}=(1-z^5)H_{n,j}$ gives
$h_{n,j}(k)-h_{n,j}(k-5)=M_{n,j}(k)-M_{n,j}(k-1)$.
Iterating towards negative indices terminates by finite Laurent
support, proving the equality; Theorem~\ref{thm:main} gives the inequality.
\end{proof}

Summing over $j$ gives the unrestricted positivity assertion
attributed in~\cite{Stanton} to Bringmann, Ono, and Rolen.

\subsection{Nine-color convolution and finite differences}

For $m\in\Nzero$, let $p^{(9)}(m)$ count ordered nine-tuples of
partitions of total size $m$, so $\sum_{m\ge0}p^{(9)}(m)x^m=(x;x)_\infty^{-9}$.
Set $G_{d,j}(z)=[q^du^j]\GG_9(u,z,q)$; these Laurent polynomials
are nonnegative and vanish for $d\le0$.

\begin{corollary}\label{cor:convolution}
For $n\in\Nzero$ and $j\in\Z$,
\begin{equation}\label{eq:convolution}
 H_{n,j}(z)=\sum_{m=0}^{\lfloor n/2\rfloor}
                    p^{(9)}(m)G_{n+1-2m,j}(z).
\end{equation}
\end{corollary}

\begin{proof}
Expand $D(q)^{-9}=\sum_{m\ge0}p^{(9)}(m)q^{2m}$ in
\eqref{eq:G9-def} and compare coefficients, using $G_{d,j}=0$ for $d\le0$.
\end{proof}

\begin{corollary}\label{cor:finite-differences}
Set $H_{n,j}=0$ for $n<0$. For $n\in\Nzero$, $j\in\Z$,
$d\ge1$, and $0\le r\le9$, with $d,r$ integral,
\begin{equation}\label{eq:finite-differences}
 \sum_{v=0}^r(-1)^v\binom rv H_{n-2dv,j}(z)
                    \in\Nzero[z,z^{-1}].
\end{equation}
More generally, for any integer sequence $(r_d)_{d\ge1}$
with $0\le r_d\le9$, including infinite support,
\begin{equation}\label{eq:mixed-differences}
 \prod_{d\ge1}(1-q^{2d})^{r_d}
       \frac{\Psi(u,z,q)}{[5]_z}
       \in\Nzero[u,u^{-1},z,z^{-1}][[q]].
\end{equation}
\end{corollary}

\begin{proof}
By \eqref{eq:G9-def}, the left side of \eqref{eq:mixed-differences} is
\[
 \GG_9(u,z,q)\prod_{d\ge1}(1-q^{2d})^{-(9-r_d)}.
\]
Every factor is nonnegative, including exponent zero.
For fixed degree $N$, factors with $2d>N$ equal $1$ modulo
$q^{N+1}$, so both infinite products are well defined.
Setting $r_d=r$ at one size and zero elsewhere, then extracting
$[q^{n+1}u^j]$, gives \eqref{eq:finite-differences}.
\end{proof}

For example,
\[
 H_{n+2d,j}(z)\ge H_{n,j}(z),\qquad
 H_{n+4d,j}(z)-2H_{n+2d,j}(z)+H_{n,j}(z)\ge0.
\]
These hold at every integer crank exponent and preserve BG parity
because the shifts in $n$ are even.

\subsection{Small cases and the role of quotient partitions}

\begin{example}\label{rem:boundary}
The partitions of $4$ are $(4)$, $(3,1)$, $(2,2)$, $(2,1,1)$, and
$(1,1,1,1)$. They all have BG-rank $0$, and their crank values are
$3,2,4,1,0$, respectively. Thus $H_{0,0}(z)=1$, and all other
$H_{0,j}$ vanish.
\end{example}

\begin{example}\label{ex:cores-only}
The positivity assertion cannot be restricted to the five-cores of a
fixed size. The five cores of size $9$ are
\[
 (6,2,1),\quad(5,2,2),\quad(5,1,1,1,1),\quad
 (3,3,1,1,1),\quad(3,2,1,1,1,1).
\]
They all have BG-rank $1$, and their crank values in this order are
$2,-2,-1,5,6$. Their generating polynomial is therefore
\begin{equation}\label{eq:cores-only}
 z^{-2}+z^{-1}+z^2+z^5+z^6
       =[5]_z(z^{-2}-1+z^2).
\end{equation}
For an algebraic verification that the list is exhaustive,
$Q(a)=2$ means $\sum_i(a_i-a_{i+1})^2=4$. The differences have sum
zero, so they consist of two $1$'s, two $-1$'s, and a zero.
To enumerate them, write $d_i=a_i-a_{i+1}$ and rotate so that $d_4=0$.
Then
$5a_0=1+4d_0+3d_1+2d_2+d_3$.
If $W$ is the sum of the two weights in $\{4,3,2,1\}$ at which
$d_i=1$, integrality requires $1-10+2W\equiv0\pmod5$.
Since $3\le W\le7$, this forces $W=7$ and hence
$(d_0,d_1,d_2,d_3,d_4)=(1,1,-1,-1,0)$.
It gives $a=(1,0,-1,0,1)$. Rotating back gives exactly the five
rotations of $(-1,0,1,1,0)$. Their edge-sign count
is $4$, their core BG-rank is $1$, and their $L$-values are precisely
the five exponents in~\eqref{eq:cores-only}.

For partitions of size $9$ with nonempty five-quotient, the core size
must instead be $4$, and the total quotient size is $1$.
The $Q=1$ vectors are the five coordinate unit vectors. Each has
three positive and two negative quotient signs. Placing the unique
one-cell quotient on a positive runner gives BG-rank $1$; placing
it on a negative runner gives BG-rank $-1$.
The resulting contributions are $3[5]_z$ and $2[5]_z$, respectively.
Combining them with~\eqref{eq:cores-only} gives
\begin{equation}\label{eq:size-nine}
 H_{1,1}(z)=z^{-2}+2+z^2,
 \qquad H_{1,-1}(z)=2.
\end{equation}
The quotient contributions cancel the negative term in
\eqref{eq:cores-only}. The resulting quotient need not be unimodal:
$H_{1,1}$ has coefficient sequence $1,0,2,0,1$ between its extreme
exponents.
\end{example}

\begin{example}\label{ex:fourteen}
Write $Z_r=z^r+z^{-r}$ for $r\ge1$. At size $14$, the nonzero
quotients are
\[
 H_{2,-2}=1,\qquad H_{2,2}=Z_2+2,\qquad
 H_{2,0}=Z_3+3Z_2+Z_1+12.
\]
For a direct calculation, the core vectors with $Q\le3$ have the
following orbits under rotation. The last column sums over all
orbits listed in its row:
\[
\begin{array}{c|c|c}
 Q&\text{orbit representatives}&\text{core crank polynomial}/[5]_z\\\hline
 1&(1,0,0,0,0)&1\\
 2&(-1,0,1,1,0)&Z_2-1\\
 3&(1,1,-1,0,0),\ (-1,1,1,0,0)&Z_3+Z_1-Z_2.
\end{array}
\]
The first two rows were determined above. For $Q=3$, the cyclic
differences have squared norm six and sum zero, hence are a
permutation of $(2,-1,-1,0,0)$ or its negative. Integrality of
$a_0$ leaves the two displayed orbits. Their core BG-rank is zero.
The $Q=1$ weight $\PP_+^3\PP_-^2$ has quotient-degree-two
coefficients $1,16,3$ at BG values $-2,0,2$.
The $Q=2$ weight $u\PP_+\PP_-^4$ has quotient-degree-one
coefficients $4,1$ at BG values $0,2$.
Combining the three rows gives the stated formulas. In particular,
$H_{2,0}-8H_{0,0}=Z_3+3Z_2+Z_1+4\ge0$, as required by
$[q^3u^0]$ in \eqref{eq:strong}.
\end{example}

\subsection{A coefficient interpretation}

The coefficients in \eqref{eq:positive-lattice} count choices of
$b\in\A$ with $b_4>0$, an integer $0\le v<b_4$, theta indices,
and the colored partitions in \eqref{eq:even-positive} or
\eqref{eq:odd-factor}--\eqref{eq:kernel-positive}, with $\ell=2$.
The lattice choice contributes $q^{Q(b)}z^{L(\rho b)+5v}$,
and each colored part of size $d$ contributes $q^{2d}$.
In the odd case, the kernel index $J$ contributes BG exponent
$\delta J$, where $\delta=\sigma_1(Tb)=\sigma_3(Tb)$.
The forbidden colors and all remaining weights are specified
in those formulas. The coordinate bound and quadratic theta
degrees make each coefficient count finite.

Adjoining an ordered eight-tuple of partitions, each part again
contributing twice its size to the $q$-degree, gives $H_{n,j}$
by \eqref{eq:H-generating}. The identity $K_{n,j}=[5]_zH_{n,j}$
groups these coefficients into blocks of five consecutive crank
exponents; it does not give a natural five-cycle on Young diagrams.
The nine-color formula \eqref{eq:convolution} uses the additional
compensation in Section~\ref{sec:nine-factors}. By Proposition~\ref{prop:ten-core},
$[5]_z\GG=\Omega/D^2$ counts a ten-core and an ordered pair of
partitions, each quotient cell adding ten cells to the total size.

\appendix
\section{Rational coefficient identities for the first local inequality}\label{app:l1-certificates}

This appendix gives the rational identities used in
Section~\ref{sec:local-L1}. We first recall the coefficient formula
for each residue class and its starting degree; see
\cite[Chapter~4]{StanleyEC}. We then verify the required identities
by clearing denominators.

\subsection{Residue-class coefficient polynomials}

\begin{lemma}[Residue-class coefficient formula]\label{lem:residue-tail}
Let $r\ge1$, let $d_1,\ldots,d_r$ be positive integers, and let
$L$ be a positive common multiple of them. For $N(x)\in\Z[x]$ put
\begin{equation}\label{eq:residue-clearing}
 F(x)=\frac{N(x)}{\prod_{i=1}^r(1-x^{d_i})},\qquad
 A(x)=N(x)\prod_{i=1}^r
             (1+x^{d_i}+\cdots+x^{L-d_i}).
\end{equation}
For $0\le s<L$, write $a_{s,j}=[x^{s+Lj}]A(x)$ for $j\ge0$. If this
residue class is nonzero, let $J_s$ be the largest $j$ for which
$a_{s,j}\ne0$. Then, for every integer
$m\ge\max\{0,J_s-r+1\}$,
\begin{equation}\label{eq:residue-tail-formula}
 [x^{s+Lm}]F(x)
   =\sum_{j=0}^{J_s}a_{s,j}
                 \binom{m-j+r-1}{r-1}.
\end{equation}
Here $\binom{y}{r-1}=y(y-1)\cdots(y-r+2)/(r-1)!$ is
interpreted as a polynomial, with $\binom y0=1$. Thus the right
side is a polynomial in $m$ of degree at most $r-1$. An empty
residue class has coefficient identically zero. If $A\ne0$, the
uniform cutoff
\begin{equation}\label{eq:residue-tail-cutoff}
 m\ge\max\{0,\lfloor\deg A/L\rfloor-r+1\}
\end{equation}
suffices for all residue classes. No sign hypothesis on $N$ or
$A$ is required.
\end{lemma}
\begin{proof}
Multiplication in \eqref{eq:residue-clearing} gives the complete
identity $F=A/(1-x^L)^r$. Expanding this denominator in formal
power series first yields
\[
 [x^{s+Lm}]F
   =\sum_{j=0}^{\min\{m,J_s\}}a_{s,j}
                     \binom{m-j+r-1}{r-1}.
\]
Suppose $m\ge J_s-r+1$. If $j>m$ and $j\le J_s$, then
$0\le m-j+r-1\le r-2$; hence the binomial polynomial at that
integer is zero. For $r=1$ there are no such $j$, since then
$m\ge J_s$. We may therefore include all $j\le J_s$, proving
\eqref{eq:residue-tail-formula}. Each binomial polynomial has degree
at most $r-1$, and $J_s\le\lfloor\deg A/L\rfloor$ gives the
uniform cutoff.
If a residue class of $A$ is empty, the initial coefficient formula
shows that the corresponding class of $F$ vanishes.
\end{proof}

For the even kernel below, $r=3$, $L=6$, and $\deg A\le95$,
so $m\ge13$ suffices and the tail starts at degrees $78+s$.
For each odd kernel, cancel its factor $1-x$ first. The remaining
denominator has $r=2$, $L=6$, and $\deg A\le29$; thus
$m\ge3$ suffices and the tails start at $18+s$.
The residue-$12$ identity in
\eqref{eq:l2-rational-certificate} has $r=3$ and
$\deg A=32$, giving cutoff zero. Thus the even and odd tails have degree at most two and one,
respectively. The tables below specify their coefficients.

\subsection{The even-slice identity}
Retain the polynomial $P_*$ from
\eqref{eq:l1-even-finite-numerator}, put $y=x^6$, and define
$C_r(y)=\sum_{n=0}^{12}c_{r,n}y^n$. Then
\begin{equation}\label{eq:l1-even-certificate}
 \frac{P_*}{(1-x^2)^2(1-x^3)}
 =\sum_{r=0}^5x^rC_r(y)
  +\sum_{r=0}^5x^{78+r}
   \left\{\frac{\alpha_r}{1-y}
    +\frac{\beta_ry}{(1-y)^2}+\frac{27y^2}{(1-y)^3}\right\}.
\end{equation}
The coefficients are
\begin{center}\small
\begin{tabular}{r|rrrrrrr}
$r$&$c_{r,0}$&$c_{r,1}$&$c_{r,2}$&$c_{r,3}$&$c_{r,4}$&$c_{r,5}$&$c_{r,6}$\\\hline
0 & 1 & 27 & 79 & 157 & 267 & 403 & 564\\
1 & 1 & 20 & 67 & 142 & 244 & 375 & 525\\
2 & 10 & 50 & 116 & 206 & 324 & 471 & 643\\
3 & -1 & 24 & 75 & 158 & 267 & 401 & 556\\
4 & 20 & 66 & 140 & 242 & 371 & 529 & 706\\
5 & 13 & 52 & 115 & 206 & 326 & 471 & 641\\
\end{tabular}
\medskip

\begin{tabular}{r|rrrrrr|rr}
$r$&$c_{r,7}$&$c_{r,8}$&$c_{r,9}$&$c_{r,10}$&$c_{r,11}$&$c_{r,12}$&$\alpha_r$&$\beta_r$\\\hline
0 & 746 & 965 & 1211 & 1483 & 1773 & 2099 & 2453 & 379\\
1 & 706 & 917 & 1156 & 1413 & 1700 & 2019 & 2362 & 370\\
2 & 840 & 1070 & 1326 & 1604 & 1906 & 2244 & 2606 & 388\\
3 & 748 & 967 & 1212 & 1475 & 1774 & 2101 & 2453 & 379\\
4 & 914 & 1152 & 1418 & 1702 & 2016 & 2362 & 2732 & 397\\
5 & 844 & 1073 & 1324 & 1599 & 1910 & 2245 & 2606 & 388\\
\end{tabular}
\end{center}
Clearing the denominator in~\eqref{eq:l1-even-certificate} gives
\begin{align}\label{eq:l1-even-cleared}
 P_*(1+x^2+x^4)^2(1+x^3)
 ={}&\sum_{r=0}^5x^r(1-x^6)^3C_r(x^6)\nonumber\\
 &+\sum_{r=0}^5x^{78+r}
 \{\alpha_r(1-x^6)^2+\beta_rx^6(1-x^6)+27x^{12}\}.
\end{align}
Both sides are polynomials of degree at most $95$.
Substitution of $P_*=1+8x^2+x(t_1+t_2+t_4)^2$ and the displayed
table verifies the identity by polynomial multiplication. Dividing by $(1-x^6)^3$ proves
\eqref{eq:l1-even-certificate} in the formal series ring.
The coefficient of the tail at degree $78+r+6m$ is
\[
 \alpha_r+\beta_rm+27\binom m2\qquad(m\ge0).
\]
Every such coefficient is positive. The prefix has only the negative
entry $c_{3,0}=-1$, and every entry at a degree at least $4$ is at least
$2$. This proves all assertions in
\eqref{eq:l1-even-rational-bounds}.

\subsection{The five odd-slice identities}
Let $N_i$ denote the numerator of the kernel indexed by
$i\in\{H_1,H_0,I,J_0,H\}$ in \eqref{eq:l1-five-kernels}. Define
\[
 Q_i(x)=\frac{N_i(x)}{(1-x)(1-x^2)(1-x^3)},\qquad y=x^6.
\]
All five numerators are divisible by $1-x$. One has
\begin{equation}\label{eq:l1-odd-certificate}
 Q_i(x)=\sum_{n=0}^{17}d_{i,n}x^n
   +\sum_{r=0}^5x^{18+r}
     \left\{\frac{a_{i,r}}{1-y}+\frac{s_i y}{(1-y)^2}\right\}.
\end{equation}
The prefix coefficients are
\begin{center}\small
\begin{tabular}{c|rrrrrrrrr}
$i$&$d_{i,0}$&$d_{i,1}$&$d_{i,2}$&$d_{i,3}$&$d_{i,4}$&$d_{i,5}$&$d_{i,6}$&$d_{i,7}$&$d_{i,8}$\\\hline
$H_1$ & 1 & 3 & 2 & 1 & 9 & 6 & 5 & 11 & 8\\
$H_0$ & 1 & 5 & 0 & 3 & 9 & 6 & 5 & 13 & 6\\
$I$ & 3 & 5 & 2 & 7 & 6 & 5 & 11 & 12 & 7\\
$J_0$ & 1 & 0 & 2 & 4 & 1 & 2 & 3 & 3 & 5\\
$H$ & 2 & 1 & 5 & 8 & 3 & 4 & 5 & 9 & 12\\
\end{tabular}
\medskip

\begin{tabular}{c|rrrrrrrrr}
$i$&$d_{i,9}$&$d_{i,10}$&$d_{i,11}$&$d_{i,12}$&$d_{i,13}$&$d_{i,14}$&$d_{i,15}$&$d_{i,16}$&$d_{i,17}$\\\hline
$H_1$ & 5 & 15 & 10 & 10 & 18 & 14 & 12 & 22 & 16\\
$H_0$ & 7 & 15 & 10 & 10 & 20 & 12 & 14 & 22 & 16\\
$I$ & 15 & 10 & 7 & 16 & 13 & 11 & 22 & 16 & 16\\
$J_0$ & 7 & 4 & 7 & 8 & 4 & 6 & 10 & 5 & 8\\
$H$ & 13 & 11 & 12 & 11 & 10 & 9 & 17 & 14 & 14\\
\end{tabular}
\end{center}
The tail data are
\begin{center}\small
\begin{tabular}{c|rrrrrr|r}
$i$&$a_{i,0}$&$a_{i,1}$&$a_{i,2}$&$a_{i,3}$&$a_{i,4}$&$a_{i,5}$&$s_i$\\\hline
$H_1$ & 16 & 24 & 20 & 18 & 28 & 22 & 6\\
$H_0$ & 16 & 26 & 18 & 20 & 28 & 22 & 6\\
$I$ & 24 & 20 & 18 & 28 & 22 & 22 & 6\\
$J_0$ & 12 & 8 & 9 & 14 & 9 & 11 & 3\\
$H$ & 22 & 20 & 16 & 26 & 22 & 20 & 6\\
\end{tabular}
\end{center}
For each row, the equivalent finite polynomial identity is
\begin{align}\label{eq:l1-odd-cleared}
 \frac{N_i(x)}{1-x}(1+x^2+x^4)(1+x^3)
 ={}&(1-x^6)^2\sum_{n=0}^{17}d_{i,n}x^n\nonumber\\
 &+\sum_{r=0}^5x^{18+r}\{a_{i,r}(1-x^6)+s_ix^6\}.
\end{align}
The quotient on the left is a polynomial; both sides have degree at
most $29$. Substituting the three polynomials in
\eqref{eq:l1-odd-bad-blocks} and the five explicit numerators in
\eqref{eq:l1-five-kernels}, direct multiplication verifies all five
identities. Division by $(1-x^6)^2$ proves
\eqref{eq:l1-odd-certificate} at every degree. Its tail coefficient at
$18+r+6m$ is $a_{i,r}+s_i m$ for $m\ge0$.

Every table entry is nonnegative. The $H_0$ row has precisely one zero,
at degree $2$; all other degrees have coefficient at least $1$.
The $H$ row has coefficient at least $1$ in every degree and coefficient
$3$ at degree $4$. The $H_1$ constant coefficient is $1$.
To recover the full kernels with denominator $E$, multiply $Q_i$ by
$\prod_{k\ge4}(1-x^k)^{-1}$, a nonnegative series with constant term
$1$ and no terms of degrees $1,2,3$. Thus these assertions give
\eqref{eq:l1-five-capacities}, including the exact zero at degree $2$
of $H_0$.

\section{Frobenius coordinates for the joint generating function}
\label{app:frobenius}

For coefficient computation, we specialize the classical Frobenius
constant-term method to BG-rank and integer crank
\cite{Andrews,BN26,CKY25}; for periodic content weights,
see~\cite[Section~3]{ChoFPSAC25}, the conference version of~\cite{Cho25}.

For a partition $\lambda$ with Durfee number
\[
 d=\max\bigl(\{i\ge1:\lambda_i\ge i\}\cup\{0\}\bigr),
\]
its \emph{Frobenius coordinates} are the arm and leg lengths
\[
 \alpha_i=\lambda_i-i,\qquad \beta_i=\lambda'_i-i
 \qquad(1\le i\le d).
\]
These are strictly decreasing nonnegative sequences; conversely,
any two such sequences of equal length determine a unique partition.
Let $\eta$ be the $5$-periodic sequence
\[
 (\eta_0,\eta_1,\eta_2,\eta_3,\eta_4)=(0,1,1,0,-2).
\]

\begin{lemma}[Frobenius specialization]\label{lem:frobenius-weights}
For every partition, including $\varnothing$,
\begin{align}
 |\lambda|&=\sum_{i=1}^d(\alpha_i+\beta_i+1),
       \label{eq:frobenius-size}\\
 \BG(\lambda)&=\sum_{i=1}^d
       \bigl(\mathbf1_{\alpha_i\text{ even}}
             -\mathbf1_{\beta_i\text{ odd}}\bigr),
       \label{eq:frobenius-bg}\\
 \cc(\lambda)&=1+\sum_{i=1}^d
       (\eta_{\alpha_i}-\eta_{-\beta_i-1}).
       \label{eq:frobenius-crank}
\end{align}
Consequently, for independent formal variables $t,v,u,z$,
\begin{equation}\label{eq:frobenius-product}
\begin{split}
 &\sum_\lambda t^{|\lambda|}u^{\BG(\lambda)}z^{\cc(\lambda)}\\
 &\quad=z[v^0]\left\{
 \prod_{a\ge0}(1+v\,u^{\mathbf1_{a\text{ even}}}
                    z^{\eta_a}t^{a+1})
 \prod_{b\ge0}(1+v^{-1}u^{-\mathbf1_{b\text{ odd}}}
                    z^{-\eta_{-b-1}}t^b)\right\}.
\end{split}
\end{equation}
\end{lemma}

\begin{proof}
The $i$th diagonal hook has $\alpha_i+\beta_i+1$ cells.
Its horizontal segment, including the diagonal cell, has BG sum
$\mathbf1_{\alpha_i\text{ even}}$; its remaining vertical segment
has sum $-\mathbf1_{\beta_i\text{ odd}}$. Summing the disjoint
hooks proves \eqref{eq:frobenius-size} and \eqref{eq:frobenius-bg}.

In \eqref{eq:crank-def}, a cell of content $r$ has the $5$-periodic
weight $\gamma_r$, with
\[
 (\gamma_0,\gamma_1,\gamma_2,\gamma_3,\gamma_4)=(2,1,0,-1,-2).
\]
Since these sum to zero, checking five residues gives, for $a,b\ge0$,
\[
 \sum_{r=0}^{a}\gamma_r=2+\eta_a,\qquad
 \sum_{r=1}^{b}\gamma_{-r}=-2-\eta_{-b-1}.
\]
The second identity includes $b=0$, since $\eta_{-1}=-2$.
Thus a diagonal hook contributes
$\eta_{\alpha_i}-\eta_{-\beta_i-1}$; the constant $1$ in the crank
gives \eqref{eq:frobenius-crank}.

The two products choose distinct arm and leg lengths, and $[v^0]$
equates their cardinalities. Sorting gives the unique Frobenius
coordinates, with weights \eqref{eq:frobenius-size}--\eqref{eq:frobenius-crank}
and prefactor $z$ for the crank constant.
Only the nonconstant term of the $b=0$ factor has $t$-degree zero;
the nonconstant terms of all other factors have positive degree.
The product therefore lies in
$\Z[v^{\pm1},u^{\pm1},z^{\pm1}][[t]]$, so the constant term is legitimate.
\end{proof}

The constant term in $t$ is $z$, the weight of $\varnothing$.
To compute through $t^N$ ($N\ge1$), retain only $0\le a,b\le N-1$:
an arm $a\ge N$ costs more than $N$, while a leg $b\ge N$
requires an arm of positive degree in $[v^0]$.
Selecting $t^m$ with $m\equiv4\pmod5$ and replacing it by
$q^{(m+1)/5}$ gives the corresponding part of $\Psi$.
The positive quotient with nine Euler factors is supplied by the
lattice argument above.

\section*{Declaration of generative AI and AI-assisted technologies
in the manuscript preparation process}

During the preparation of this work, the authors used AI-assisted tools
to assist with language editing, manuscript organization,
proof checking, and symbolic verification.
The authors reviewed and edited the resulting material and take
full responsibility for the contents of the manuscript.

\end{document}